\documentclass[12pt]{amsart}

\usepackage{amssymb,amsthm,amssymb}
\usepackage{graphicx,color}  
\DeclareGraphicsExtensions{.pstex,.eps}

\newcommand{\supp}{\operatorname{supp}}

\renewcommand{\Re}{\mathop{\rm Re}\nolimits}
\renewcommand{\Im}{\mathop{\rm Im}\nolimits}

\theoremstyle{plain}

\newtheorem{thm}{Theorem}
\newtheorem{prop}{Proposition}[section]

\newtheorem{lem}[prop]{Lemma}

\theoremstyle{definition}

\newtheorem{rem}{Remark}[section]
\newtheorem{defn}[prop]{Definition} 

\numberwithin{equation}{section}

\def\squarebox#1{\hbox to #1{\hfill\vbox to #1{\vfill}}}

\newcommand{\la}{\langle}
\newcommand{\ra}{\rangle}

\usepackage{amsxtra}

\usepackage{cite,fullpage}
\DeclareMathOperator{\sgn}{sgn}
\newcommand{\R}{\mathbb R}
\newcommand{\K}{\mathbf k}
\DeclareMathOperator{\cosech}{cosech}
\usepackage{pgfplots}
\usetikzlibrary{patterns}
\newcommand{\Z}{\mathbb Z}
\newcommand{\Til}{\mathcal T}
\newcommand{\Hil}{\mathcal H}
\newcommand{\mc}{\mathcal}
\newcommand{\mb}{\mathbf}
\newcommand{\mr}{\mathrm}

\newcommand{\el}{\mr{ell}}
\newcommand{\hyp}{\mr{hyp}}

\newcommand{\high}{\mr{high}}
\newcommand{\low}{\mr{low}}
\newcommand{\err}{\mb{err}}
\usepackage{mathrsfs}
\newcommand{\symb}{\mathscr S}
\title
{Small data global solutions for the Camassa-Choi equations}

\author[Benjamin Harrop-Griffiths]
{Benjamin Harrop-Griffiths}
\email{benjamin.harrop-griffiths@cims.nyu.edu}
\address{Courant Institute of Mathematical Sciences, New York University\\
251 Mercer Street, New York, NY 10012, USA}
\thanks{The first author was supported by a Junior Fellow award from the Simons Foundation.}

\author[J.L. Marzuola]
{Jeremy L. Marzuola}
\email{marzuola@math.unc.edu}
\address{Mathematics Department, University of North Carolina \\
Phillips Hall, Chapel Hill, NC 27599, USA}
\thanks{The second author was supported in
  part by U.S. NSF Grants DMS--1312874 and DMS-1352353.}
\thanks{The authors thank Roberto Camassa, Mihaela Ifrim, and Daniel Tataru for many useful discussions especially at the beginning of this work in establishing the right functional framework.}

\begin{document}

\begin{abstract}
We consider solutions to the Cauchy problem for an internal-wave model derived by Camassa-Choi \cite{MR1389977}. This model is a natural generalization of the Benjamin-Ono and Intermediate Long Wave equations in the case of weak transverse effects. We prove the existence and long-time dynamics of global solutions from small, smooth, spatially localized initial data on \(\R^2\).
\end{abstract}

\maketitle

\section{Introduction}

In this article we consider real-valued solutions \(u\colon\R_t\times\R^2_{(x,y)}\rightarrow\R\) to the Cauchy problem for an internal-wave model derived by Camassa-Choi \cite{MR1389977},
\begin{equation}
\label{eqn:cc-fd}
\left( u_t + \Til_h^{-1}u_{xx}  + h^{-1}u_x - u u_x  \right)_x + u_{yy} = 0,
\end{equation}
where \(h>0\) is the depth and the operator \(\Til_h^{-1}\) has symbol \(i\coth(h\xi)\). In the limit \(h \rightarrow \infty\) we obtain the infinite depth equation,
\begin{equation}
\label{eqn:cc}
\left( u_t + \Hil^{-1}u_{xx} - u u_x \right)_x + u_{yy} = 0,
\end{equation}
where the inverse of the Hilbert transform \(\Hil^{-1} = - \Hil\) has symbol \(i\sgn\xi\). These are natural \(2\)-dimensional versions of the Intermediate Long Wave (ILW) and Benjamin-Ono equations in the case of weak transverse effects. Our goal is to investigate the long-time dynamics of solutions with sufficiently small, smooth and spatially localized initial data.

The infinite depth equation \eqref{eqn:cc} is a special case of the dispersion-generalized (or fractional) Kadomtsev-Petviashvili II (KP-II) equation,
\begin{equation}\label{eqn:gKPII}
\left( u_t - |D_x|^\alpha u_x + uu_x\right)_x + u_{yy} = 0.
\end{equation}
The original KP-II equation corresponds to the case \(\alpha = 2\) and is completely integrable in the sense that it possesses both a Lax pair and an infinite number of formally conserved quantities (see for example the survey article \cite{Klein2015}). To authors' knowledge, a Lax pair is not known for \eqref{eqn:cc-fd} or \eqref{eqn:cc} although both of their \(1d\) counterparts, the ILW and Benjamin-Ono, are integrable in this sense.

Both the finite and infinite depth equations are Hamiltonian with formally conserved energies,
\begin{gather}
E_h[u] = \int \left( u\Til_h^{-1}\partial_xu + h^{-1}u^2 + (\partial_x^{-1}\partial_yu)^2 - \frac13 u^3\right)\,dxdy,\\
E_\infty[u] = \int \left(u\Hil^{-1}\partial_xu + (\partial_x^{-1}\partial_yu)^2 - \frac13u^3\right)\,dxdy,
\end{gather}
respectively. Both equations also conserve the \(L^2\)-norm,
\begin{equation}
M[u] = \int u^2\,dxdy.
\end{equation}

The infinite depth equation is invariant with respect to the scaling
\begin{equation}\label{Scaling}
u(t,x,y)\mapsto \lambda u(\lambda^2t,\lambda x,\lambda^{\frac{3}{2}}y),\qquad \lambda>0.
\end{equation}
Taking \(\lambda = h\), this scaling also maps solutions to the finite depth equation with depth \(h\) to solutions with depth \(1\). Both the finite and infinite depth equations are invariant with respect to Galilean shifts of the form,
\begin{equation}\label{Galilean}
u(t,x,y)\mapsto u(t,x+cy - c^2t,y - 2ct),\qquad c\in\R.
\end{equation}
The natural spaces in which to consider local well-posedness for the infinite depth equation are the homogeneous anisotropic Sobolev spaces \(\dot H^{s_1,s_2} = \dot H^{s_1}_x\dot H^{s_2}_y\) with norm
\[
\|u\|_{\dot H^{s_1,s_2}} = \||D_x|^{s_1}|D_y|^{s_2}u\|_{L^2_{x,y}},
\]
for which the scaling-critical, Galilean-invariant space is given by \((s_1,s_2) = (\frac14,0)\).

Small data global well-posedness and scattering were proved for the KP-II at the scaling-critical regularity \((s_1,s_2) = (-\frac12,0)\) by Hadac-Herr-Koch \cite{MR2526409,MR2629889}. Local well-posedness results are also available in higher dimensions \cite{2016arXiv160806730K} and for the dispersion generalized equation \eqref{eqn:gKPII} on \(\R_{x,y}^2\) provided \(\alpha>\frac43\) \cite{MR2434299}. While preparing this paper we also learned of a recent result of Linares-Pilod-Saut~\cite{2017arXiv170509744L} who prove several local well-posedness and ill-posedness results for \eqref{eqn:cc-fd}, \eqref{eqn:cc} and other similar generalizations of the KP-II.

We define the linear operator
\[
\mc L_h = \partial_t + \Til_h^{-1}\partial_x^2 + h^{-1}\partial_x + \partial_x^{-1}\partial_y^2,
\]
with the corresponding modification when \(h = \infty\). Here \(\partial_x^{-1}\) is interpreted as the Fourier multiplier \(\mr{p.v. }(i\xi)^{-1}\), which for \(f\in L^1\) gives us
\[
\partial_x^{-1}f = \frac12\int_{-\infty}^xf(y)\,dy - \frac12\int_x^\infty f(y)\,dy.
\]
For \(\xi\neq0\), the dispersion relation associated to \eqref{eqn:cc-fd} is given by
\begin{equation}\label{OMEGAh}
\omega_h(\K) = \xi^2\coth(h \xi) - h^{-1}\xi - \xi^{-1}\eta^2,
\end{equation}
where \(\K = (\xi,\eta)\), and in the limit \(h = \infty\) we obtain
\begin{equation}\label{OMEGAinf}
\omega_\infty(\K) = \xi|\xi| - \xi^{-1}\eta^2.
\end{equation}

We take \(S_h(t)\) to be the associated linear propagator, defined using the Fourier transform\footnote{We use the isometric normalization of the Fourier transform, \(\hat f(\K) = \frac1{2\pi}\int f(x,y)e^{-i(x,y)\cdot\K}\,d\K\).} as
\begin{equation}\label{Propagator}
S_h(t)f = \frac{1}{2\pi}\lim\limits_{\epsilon\downarrow0}\int_{|\xi|>\epsilon} e^{it\omega_h(\K)}\hat f(\K)e^{i(x,y)\cdot\K}\,d\K,
\end{equation}
with the corresponding modification when \(h = \infty\). We note that the linear propagator \eqref{Propagator} extends to a well-defined unitary map \(S_h(t)\colon L^2(\R^2)\rightarrow L^2(\R^2)\) without the need for additional moment assumptions on \(f\).

Linear solutions satisfy the dispersive estimates,
\begin{equation}\label{est:Dispersive}
\|h^{\frac12}\la hD_x\ra^{-\frac12} S_h(t)u_0\|_{L^\infty}\lesssim |t|^{-1}\|u_0\|_{L^1},
\end{equation}
which may be readily seen from the fact that the kernel \(K_h\) of the linear propagator \(S_h\) is given by
\[
K_h(t,x,y) =  \frac1{\sqrt{2ht}}\ k\left(h^{-2}t,h^{-1}(x + \frac1{4t}y^2)\right),
\]
where the oscillatory integral
\[
k(t,x) = \frac1{\sqrt{2\pi}}\lim\limits_{\epsilon\downarrow0}\int_{|\xi|>\epsilon} |\xi|^{\frac12}e^{it(\xi^2\coth\xi - \xi) + ix\xi}\,d\xi.
\]
For a more detailed proof see \cite[Lemmas~4.7,~4.8]{2017arXiv170509744L}.

Due to the \(O(|t|^{-1})\) decay, bilinear interactions are long-range so it is natural to seek a normal form transformation that upgrades the quadratic nonlinearity to a cubic one. Resonant nonlinear interactions correspond to solutions of the system
\[
\begin{cases}
\omega_h(\K_1) + \omega_h(\K_2) + \omega_h(\K_3) = 0\vspace{0.1cm}\\
\K_1+\K_2+\K_3 = 0,
\end{cases}
\]
and some elementary algebraic manipulations show that this cannot be satisfied for non-zero \(\xi_j\). As a consequence the Camassa-Choi nonlinearity is non-resonant and formally we may construct a normal form leading to enhanced lifespan solutions.

Given the non-resonance of the bilinear interactions, one might expect the methods used for large values of \(\alpha\) to apply to the Camassa-Choi. However, in the corresponding \(1\)-dimensional cases of the ILW and Benjamin-Ono it is known that Picard iteration methods fail \cite{MR1885293,MR2172940} due to strong high-low bilinear interactions. While the additional dispersion in \(2d\) should allow for improved results over the corresponding \(1d\) case, one may apply the methods of \cite{MR1885293,2016arXiv160806730K} to show that Picard iteration still fails in the infinite depth case \eqref{eqn:cc} in the anisotropic Sobolev space \(\dot H^{\frac14,0}\) and almost all of the Besov-type refinements considered in \cite{2016arXiv160806730K}. For completeness we provide a brief proof of this ill-posedness result in Appendix~\ref{app:IP}.

Instead we will assume additional spatial localization on the initial data and establish global existence using a similar approach to work of the first author with Ifrim and Tataru on the KP-I equation \cite{2014arXiv1409.4487H}. Here we will apply the \emph{method of testing by wave packets}, originally developed in work of Ifrim-Tataru on the \(1d\) cubic NLS \cite{MR3382579} and \(2d\) gravity water waves \cite{2014arXiv1404.7583I}, and subsequently applied in several other contexts \cite{MR3462131,2014arXiv1406.5471I,2014arXiv1409.4487H}. The key difficulty we encounter when adapting this method to the setting of the Camassa-Choi is the presence of the non-local operator \(\Til_h^{-1}\). Indeed, a testament to the robust nature of this approach is that it may be applied to obtain global solutions even in this context. We note that a related approach to establishing global well-posedness is via the space-time resonances method, simultaneously developed by Germain-Masmoudi-Shatah \cite{MR2542891,MR2482120} and Guftason-Nakanishi-Tsai \cite{MR2360438} as a significant upgrade to the method of normal forms, originally applied in the context of dispersive PDE by Shatah \cite{MR803256}.

We expect that linear solutions initially localized in space near \((x,y) = 0\) and at frequency \(\K = (\xi,\eta)\) will travel along rays of the Hamiltonian flow
\[
\Gamma_{\mb v} = \{(x,y) = t\mb v\},
\]
where the group velocity
\[
\mb v = - \nabla\omega_h(\K).
\]
In order to measure this localization we define the operators \(L_{x,h}\), \(L_y\) with corresponding symbols \(x + t\partial_\xi\omega_h(\K)\), \(y + t\partial_\eta\omega_h(\K)\) respectively. A simple computation yields the explicit expressions,
\[
L_{x,h} = x - 2t\Til_h^{-1}\partial_x - th\partial_x^2(1 + \Til_h^{-2}) - th^{-1} + t\partial_x^{-2}\partial_y^2,\qquad L_y = y - 2t\partial_x^{-1}\partial_y,
\]
and in the limit \(h = \infty\),
\[
L_{x,\infty} = x - 2t\Hil^{-1}\partial_x + t\partial_x^{-2}\partial_y^2.
\]
Further, by construction the operators \(L_{x,h},L_y\) commute with \(\mc L_h\).

As the equations possesses a Galilean invariance, in the spirit of \cite{MR3382579,2014arXiv1409.4487H} we will consider well-posedness in Galilean invariant spaces. The vector field \(\partial_xL_y\) is the generator of the Galilean symmetry. However, the operators \(\partial_y\), \(L_{x,h}\) do not commute with the Galilean group so we will instead measure \(x\)-localization with the Galilean-invariant operator
\[
J_h = L_{x,h}\partial_x + L_y\partial_y.
\]
We then define the time-dependent space \(X_h\) of distributions on \(\R^2\) with finite norm,
\[
\|u\|_{X_h}^2 = h^{-\frac12}\|u\|_{L^2}^2 + h^{\frac{15}2}\|\partial_x^4u\|_{L^2}^2 + h^{-\frac92}\|L_y^2\partial_x u\|_{L^2}^2 + h^{-\frac12}\|J_hu\|_{L^2}^2.
\]
We note that this space is uniform with respect to \(h\) and that \(S_h(-t)X_h = X_h(0)\), where the initial data space
\[
\|u_0\|_{X_h(0)} = h^{-\frac12}\|u_0\|_{L^2}^2 + h^{\frac{15}2}\|\partial_x^4u_0\|_{L^2}^2 + h^{-\frac92}\|y^2\partial_xu_0\|_{L^2}^2 + h^{-\frac12}\|(x\partial_x + y\partial_y)u_0\|_{L^2}^2.
\]

Our main result in the finite depth case is the following:
\medskip
\begin{thm}\label{thm:Main}
There exists \(0<\epsilon\ll1\) so that for all \(h>0\) and \(u(0)\in X_h(0)\) satisfying the estimate
\begin{equation}\label{est:Init}
\|u(0)\|_{X_h(0)} \leq\epsilon,
\end{equation}
there exists a unique global solution \(u\in C(\R;X_h)\) to \eqref{eqn:cc-fd} satisfying the energy estimate
\begin{equation}\label{Fdn}
\|u\|_{X_h(t)}\lesssim \epsilon \la h^{-2}t\ra^{C\epsilon}
\end{equation}
and the pointwise decay estimate
\begin{equation}\label{est:PTWISEDECAY}
h^{2}\|u_x\|_{L^\infty}\lesssim \epsilon |h^{-2}t|^{-\frac12}\la h^{-2}t\ra^{-\frac12}.
\end{equation}

Further, this solution scatters in the sense that there exists some \(W\in L^2\) so that \(\|W\|_{L^2} = \|u_0\|_{L^2}\) and
\begin{equation}\label{Fds}
h^{-\frac14}\|S_h(-t)u - W\|_{L^2} \lesssim \epsilon^2 (h^{-2}t)^{-\frac1{96} + C\epsilon},\qquad t\rightarrow+\infty.
\end{equation}
\end{thm}
\medskip
\begin{rem}
As in~\cite{MR3382579,2014arXiv1409.4487H,MR3462131} we may interpolate between the estimates \eqref{Fdn} and \eqref{Fds} to show that \(W\in (X_h(0),h^{\frac14}L^2)_{C\epsilon,2}\) for some \(C>0\) where \((X_h(0),h^{\frac14}L^2)_{C\epsilon,2}\) is the usual real interpolation space.
\end{rem}
\medskip

In the infinite depth case we define the modification
\[
\|u\|_{X_\infty}^2 = \|u\|_{L^2}^2 +\|\partial_x^4u\|_{L^2}^2 + \|L_y^2\partial_x u\|_{L^2}^2 + \|J_\infty u\|_{L^2}^2,
\]
for which the initial data space is given by
\[
\|u_0\|_{X_\infty(0)} = \|u_0\|_{L^2}^2 + \|\partial_x^4u_0\|_{L^2}^2 + \|y^2\partial_xu_0\|_{L^2}^2 + \|(x\partial_x + y\partial_y)u_0\|_{L^2}^2.
\]
Our main result in the infinite depth case is then:
\medskip
\begin{thm}\label{thm:MainInf}
There exists \(0<\epsilon\ll1\) so that for all \(u(0)\in X_\infty(0)\) satisfying the estimate
\begin{equation}
\|u(0)\|_{X_\infty(0)} \leq\epsilon,
\end{equation}
there exists a unique global solution \(u\in C(\R;X_\infty)\) to \eqref{eqn:cc-fd} satisfying the energy estimate
\begin{equation}
\|u\|_{X_\infty(t)}\lesssim \epsilon\la t\ra^{C\epsilon}
\end{equation}
and the pointwise decay estimate
\begin{equation}
\|u_x\|_{L^\infty}\lesssim \epsilon |t|^{-\frac12}\la t\ra^{-\frac12}.
\end{equation}

Further, this solution scatters in the sense that there exists some \(W\in L^2\) so that
\begin{equation}
\|S_\infty(-t)u - W\|_{L^2} \lesssim \epsilon t^{-\frac1{48} + C\epsilon},\qquad t\rightarrow+\infty.
\end{equation}
\end{thm}
\medskip

\begin{rem}
The spaces \(X_h\) and \(X_\infty\) are essentially identical to the spaces considered in \cite{2014arXiv1409.4487H}. However, due to the reduced dispersion of \eqref{eqn:cc-fd} and \eqref{eqn:cc} at high frequency, we require an additional \(x\)-derivative to obtain sufficient decay.
\end{rem}

\medskip

For simplicity we now restrict our attention to the case \(h = 1\) and drop the dependence on \(h\) throughout our notation. We remark the remaining finite depth cases may be obtained by scaling.  Namely, we have
\[
u^{(1)}(t,x,y) = hu^{(h)}(h^2t,hx,h^{\frac32}y).
\]
The proof in the infinite depth case (Theorem~\ref{thm:MainInf}) is essentially identical and is thus omitted. However, for completeness we will outline a few of the required modifications when they deviate sufficiently from the finite depth proof.

The remainder of the paper is structured as follows: in Section~\ref{sect:Prelim} we give some brief technical preliminaries; in Section~\ref{sect:AP} we prove local well-posedness and a priori energy estimates for solutions to \eqref{eqn:cc-fd}; in Section~\ref{sect:KS} we prove pointwise bounds in the spirit of Klainerman-Sobolev estimates; in Section~\ref{sect:TWP} we complete the proof of Theorem~\ref{thm:Main} using the method of testing by wave packets. In Appendix~\ref{app:IP} we prove an ill-posedness result for \eqref{eqn:cc}.

%%%%%%%%%%%%%%%%%%%%%%%%%%
%%% TECHNICAL NONSENSE %%%
%%%%%%%%%%%%%%%%%%%%%%%%%%

\section{Preliminaries}\label{sect:Prelim}
In this section we briefly state several technical preliminaries required for the proof.

\subsection{The resonance function}
We define the resonance function
\[
\Omega(\K_1,\K_2,\K_3) = \omega(\K_1) + \omega(\K_2) + \omega(\K_3).
\]
If we restrict the resonance function to the hyperplane \(\{\K_1 + \K_2 + \K_3 = 0\}\) we may compute a simplified expression for 
\(
\Omega(\K_1,\K_2) = \Omega(\K_1,\K_2,-(\K_1 + \K_2)),
\)
\begin{equation}\label{ResonanceFunction}
\Omega(\K_1,\K_2) = - \xi_1\xi_2(\xi_1 + \xi_2)\left(\frac{(\xi_1 + \xi_2)^2\coth(\xi_1 + \xi_2) - \xi_1^2\coth\xi_1 - \xi_2^2\coth\xi_2}{\xi_1\xi_2(\xi_1 + \xi_2)} + \frac{\left|\dfrac{\eta_1}{\xi_1} - \dfrac{\eta_2}{\xi_2}\right|^2}{(\xi_1 + \xi_2)^2}\right).
\end{equation}

By considering the asymptotic behavior of \(\coth\xi\) we obtain the lower bound
\[
\frac{(\xi_1 + \xi_2)^2\coth(\xi_1 + \xi_2) - \xi_1^2\coth\xi_1 - \xi_2^2\coth\xi_2}{\xi_1\xi_2(\xi_1 + \xi_2)} \gtrsim \frac1{1 + \max\{|\xi_1|,|\xi_2|,|\xi_1 + \xi_2|\}},
\]
and hence we obtain lower bounds for the resonance function \(\Omega_h\) in the low-high (\(|\xi_1|\ll|\xi_2|\)) and high-high (\(|\xi_1 + \xi_2|\ll|\xi_2|\)) asymptotic regions,
\begin{equation}\label{FDResonanceLB}
|\Omega(\K_1,\K_2)|\gtrsim \begin{cases}\dfrac{|\xi_1||\xi_2|^2}{\la \xi_2\ra}\left(1 + \dfrac{\la\xi_2\ra}{|\xi_1|^2}\left|\dfrac{\eta_1 + \eta_2}{\xi_1 + \xi_2} - \dfrac{\eta_2}{\xi_2}\right|^2\right),&\quad |\xi_1|\ll|\xi_2|,\vspace{0.1cm}\\\dfrac{|\xi_1 + \xi_2||\xi_2|^2}{\la \xi_2\ra}\left(1 + \dfrac{\la\xi_2\ra }{|\xi_1 + \xi_2|^2}\left|\dfrac{\eta_1}{\xi_1} - \dfrac{\eta_2}{\xi_2}\right|^2\right),&\quad |\xi_1 + \xi_2|\ll|\xi_2|.\end{cases}
\end{equation}

The resonance function in the infinite depth case is given by the slightly more straightforward expression,
\begin{equation}\label{ResonanceInf}
\Omega_\infty(\K_1,\K_2) = \xi_1\xi_2\xi_3\left(\frac{2}{\max\{|\xi_1|,|\xi_2|,|\xi_3|\}} + \frac{\left|\dfrac{\eta_1}{\xi_1} - \dfrac{\eta_2}{\xi_2}\right|^2}{\xi_3^2}\right).
\end{equation}

\subsection{Littlewood-Paley decomposition} We take \(0<\delta\ll1\) to be a fixed universal constant that determines the resolution of our frequency decomposition. The size of \(\delta\) will be determined in the ODE estimates of Section~\ref{sect:TWP} but will otherwise be irrelevant. For \(N\in 2^{\delta \Z}\) we take \(P_N\) to project to \(x\)-frequencies \(2^{-\delta}N<|\xi|<2^\delta N\) so that
\[
1 = \sum\limits_{N\in 2^{\delta\Z}}P_N.
\]

We write \(u_N = P_Nu\), and for \(A>0\) we take \(u_{<A} = \sum_{N<A}u_N\) and \(u_{\geq A} = \sum_{N\geq A}u_N\), where the sums are understood to be over \(N\in 2^{\delta\Z}\). We observe that for any \(A>0\) we have
\[
\|u\|_X^2 \sim \|u_{<A}\|_X^2 + \sum\limits_{N\geq A}\|u_N\|_X^2.
\]

We may further decompose \(u_N = u_N^+ + u_N^-\) where \(u_N^+\) is the projection to positive wavenumbers in \(x\)-frequency. For real-valued \(u\) we have \(u_N = 2\Re(u_N^+)\) and hence
\[
\|u_N\|_X \sim \|u_N^+\|_X \sim \|u_N^-\|_X.
\]

\subsection{Symbol classes and elliptic operators}
Given \(d\geq1\) we define the symbol class \(\symb\) to consist of functions \(p\in C^\infty(\R^d\times \R^d\backslash\{0\})\) so that, writing \(x = (x_1,\dots,x_d)\) and \(\xi = (\xi_1,\dots,\xi_d)\), we have
\[
|D_x^\beta D_\xi^\alpha p(x,\xi)|\lesssim_{|\alpha|,|\beta|} |\xi|^{-|\alpha|}.
\]

Given a symbol \(p\in\symb\) we may define a pseudo-differential operator
\[
p(x,D)u = \frac1{(2\pi)^{\frac d2}}\int_{\R^d} p(x,\xi)\hat f(\xi) e^{ix\cdot\xi}\,d\xi,
\]
and then have the estimate (see for example \cite{MR1766415})
\begin{equation}
\|p(x,D)u\|_{L^2}\lesssim \|u\|_{L^2}.
\end{equation}

We also recall that if \(r(x,\xi) = p(x,\xi)q(x,\xi)\) for \(p,q\in \symb\) then we have the estimate
\begin{equation}\label{ProductRule}
\|r(x,D)f\|_{L^2} \lesssim \|p(x,D)q(x,D)f\|_{L^2} + \|f\|_{H^{-1}}.
\end{equation}
Further, if \(p\in\symb\) satisfies the estimate  (again see \cite{MR1766415})
\[
|p(x,\xi)^{-1}|\lesssim 1,
\]
then we say \(p\) is \emph{elliptic} and for \(f\in L^2(\R^d)\) and any \(s\in \R\) we have the estimate
\begin{equation}\label{OPSElliptic}
\|f\|_{L^2} \lesssim_s \|p(x,D)f\|_{L^2} + \|f\|_{H^s}.
\end{equation}

\subsection{Multilinear Fourier multipliers}
If \(d = 2n\) and \(m(\K_1,\dots,\K_n)\in\symb\) is independent of the spatial variables, we may define a multilinear Fourier multiplier \(M\) with symbol \(m\) by
\[
M[u_1,\dots,u_n] = \frac1{(2\pi)^n}\int_{\R^{2n}} m(\K_1,\dots,\K_n) \hat u_1(\K_1)\dots\hat u_n(\K_n) e^{i(x,y)\cdot(\K_1 + \dots + \K_n)}\,d\K_1\dots d\K_n.
\]
We then recall the Coifman-Meyer Theorem (see for example~\cite{MuscaluPipherTaoThiele} and references therein)
\begin{equation}\label{CM}
\|M[u_1,\dots,u_n]\|_{L^p}\lesssim \|u_1\|_{L^{p_1}}\dots\|u_n\|_{L^{p_n}},
\end{equation}
provided \(\frac 1p = \frac1{p_1} + \dots + \frac1{p_n}\),  \(1\leq p <\infty\), \(1<p_j\leq\infty\).

\subsection{Sobolev estimates}
We recall the Sobolev estimate,
\begin{equation}
\|f\|_{L^\infty} \lesssim \|f\|_{L^2}^{\frac14}\|f_x\|_{L^2}^{\frac12}\|f_{yy}\|_{L^2}^{\frac14},\label{est:Sobolev}
\end{equation}
and the H\"older space modification,
\begin{equation}
\|f\|_{\dot C^{0,\alpha}} \lesssim \left(\|f\|_{L^2}^{\frac14 - \alpha}\|f_x\|_{L^2}^\alpha + \|f\|_{L^2}^{\frac14 - \frac\alpha 2}\|f_{yy}\|_{L^2}^{\frac\alpha2}\right)\|f_x\|_{L^2}^{\frac12}\|f_{yy}\|_{L^2}^{\frac14},\qquad 0<\alpha\leq\frac14.\label{est:Holder}
\end{equation}

%%%%%%%%%%%%%%%%%%%%%%%%
%%% ENERGY ESTIMATES %%%
%%%%%%%%%%%%%%%%%%%%%%%%

\section{Local well-posedness and energy estimates}\label{sect:AP}
In this section we prove a priori estimates for the solution to \eqref{eqn:cc-fd} in the case \(h = 1\). As a consequence of the usual energy method we obtain local well-posedness in the spaces \(Z^k\), where the norm
\begin{equation}
\|u\|_{Z^k}^2 = \|u\|_{L^2}^2 + \|\partial_x^ku\|_{L^2}^2 + \|L_y^2\partial_x u\|_{L^2}^2,
\end{equation}
and we note that \(X\subset Z^4\). Our local well-posedness result is the following:
\medskip
\begin{thm}\label{thm:LWP}
Let \(h = 1\) and \(k\geq3\). Then for all \(u_0\in Z^k(0)\), the equation \eqref{eqn:cc-fd} is locally well-posed in \(Z^k\) and the solution exists at least as long as \(\left|\int_0^t \|u_x(s)\|_{L^\infty}\,ds\right| <\infty\).
\end{thm}
\medskip
\begin{rem}
Our definition of local well-posedness in Theorem~\ref{thm:LWP} includes:
\begin{itemize}
\item \emph{Existence.} There exists a solution \(S(-t)u\in C([-T,T];Z^k(0))\).
\item \emph{Uniqueness.} The solution \(S(-t)u\) is unique in the space \(Z^k(0)\).
\item \emph{Continuity.} The solution map \(u_0 \mapsto S(-t)u(t)\) is continuous in the \(Z^k(0)\) topology.
\item \emph{Persistence of regularity.} If \(u_0\in Z^{k'}(0)\) for \(k'\geq k\) then \(u\in Z^{k'}\).
\end{itemize}
\end{rem}
\medskip
\begin{rem}
The result of Theorem~\ref{thm:LWP} is certainly not optimal in terms of regularity but will suffice for the purposes of establishing the existence of global solutions. Indeed, an elementary application of the usual Littlewood-Paley trichotomy allows us to obtain local well-posedness in \(Z^{\frac 72+}\). We also mention several other local well-posedness results in other topologies are proved in~\cite{2017arXiv170509744L}.
\end{rem}
\medskip
\begin{rem}
As usual, it suffices to consider times \(t\geq0\) as the equation is invariant under the transformation
\(
u(t,x,y)\mapsto u(-t,-x,y).
\)
\end{rem}
\medskip

The key ingredient for local well-posedness will be a priori estimates for the solution in the space \(Z^k\). We will supplement these a priori bounds with a further estimate when the initial data \(u_0\in X\) satisfies the smallness condition \eqref{est:Init} and obtain energy estimates for the solution depending upon the size of
\begin{equation}\label{ControlQuantity}
\mc M = \frac1\epsilon\sup\limits_{t\in[0,T]} |t|^{\frac12}\la t\ra^{\frac12}\|u_x\|_{L^\infty}
\end{equation}
Our main a priori bound is the following:
\medskip
\begin{prop}\label{prop:NRG}
Let \(h=1\) and \(u\) be a smooth solution to \eqref{eqn:cc-fd} on the time interval \([0,T]\). Then we have the a priori estimate,
\begin{equation}\label{est:Zk}
\|u(t)\|_{Z^k} \leq \|u_0\|_{Z^k(0)} \exp\left(C\int_0^t \|u_x(s)\|_{L^\infty}\,ds\right).
\end{equation}

Further, if the initial data \(u_0\in X(0)\) satisfies \eqref{est:Init} and \(0<\epsilon\ll \mc M^{-1}\) is sufficiently small, we have the improved estimate
\begin{equation}\label{est:NRG}
\|u(t)\|_X\lesssim \epsilon \la t\ra^{C\mc M\epsilon}.
\end{equation}
\end{prop}
\begin{proof}
In order to justify the various computations we note that by standard approximation arguments it suffices to assume that \(u\) is a Schwartz function and that for some \(0<\nu\ll1\) we have \(P_{< \nu}u = 0\).

\emph{Estimates for \(\partial_x^ku\).}
Differentiating the equation \(k\) times we obtain
\[
\mc L\partial_x^ku = u\partial_x^{k+1}u + \frac12\sum\limits_{j=1}^k\binom{k+1}{j}\partial_x^ju\cdot\partial_x^{k+1-j}u.
\]
Integrating by parts for the first term and using the elementary interpolation estimate
\[
\|\partial_x^ju\|_{L^{\frac{2(k-1)}{j-1}}}^{k-1}\lesssim \|u_x\|_{L^\infty}^{k-j}\|\partial_x^ku\|_{L^2}^{j - 1},\qquad 1\leq j\leq k,
\]
for the second term we obtain the bound
\begin{equation}\label{est:D4u}
\frac{d}{dt}\|\partial_x^ku\|_{L^2}^2\lesssim \|u_x\|_{L^\infty}\|\partial_x^ku\|_{L^2}^2.
\end{equation}

\smallskip
\emph{Estimates for \(L_y^2\partial_xu\).} Again we start by calculating
\[
\mc LL_y^2\partial_xu = uL_y^2\partial_x^2u + (L_y\partial_xu)^2.
\]
For the first term we may simply integrate by parts. For the second term it suffices to show that
\begin{equation}\label{L4Bound}
\|L_yu_x\|_{L^4}^2\lesssim \|u_x\|_{L^\infty}\|L_y^2u_x\|_{L^2},
\end{equation}
from which we obtain the estimate
\begin{equation}\label{est:Lyu}
\frac{d}{dt}\|L_y^2\partial_xu\|_{L^2}^2\lesssim \|u_x\|_{L^\infty}\|L_y^2\partial_xu\|_{L^2}^2.
\end{equation}

To prove the estimate \eqref{L4Bound} we first make a change of variables \(f(t,x,y) = u(t,x - \frac1{4t}y^2,y)\) so that
\[
\|f_x\|_{L^\infty} = \|u_x\|_{L^\infty},\qquad\|f_y\|_{L^4} = \|L_y\partial_xu\|_{L^4},\qquad \|\partial_x^{-1}f_{yy}\|_{L^2} = \|L_y^2\partial_xu\|_{L^2}.
\]
We first observe that by symmetry we may decompose by frequency as
\[
\|f_y\|_{L^4}^4 = \sum\limits_{N_1\sim N_2}\int f_{N_1,y} f_{N_2,y} (f_{\leq N_2,y})^2\,dxdy.
\]
We then integrate by parts to obtain
\begin{align*}
\|f_y\|_{L^4}^4 &= - \sum\limits_{N_1\sim N_2}\int f_{N_1} f_{N_2,yy} (f_{\leq N_2,y})^2\,dxdy - 2\sum\limits_{N_1\sim N_2}\int f_{N_1} f_{N_2,y} f_{\leq N_2,y} f_{\leq N_2,yy}\,dxdy\\
&\lesssim \|f_x\|_{L^\infty}\left(\|C_1[\partial_x^{-1}f_{yy},f_y,f_y]\|_{L^1} + \|C_2[f_y,f_y,\partial_x^{-1}f_{yy}]\|_{L^1}\right),
\end{align*}
where the trilinear Fourier multipliers
\[
C_1[f,g,h] = \sum\limits_{N_1\sim N_2}\partial_x^{-1}P_{N_1}(\partial_xf_{N_2} g_{\leq N_2}h_{\leq N_2}),\quad C_2[f,g,h] = \sum\limits_{N_1\sim N_2}\partial_x^{-1}P_{N_1}(f_{N_2}g_{\leq N_2}\partial_xh_{\leq N_2}),
\]
may be bounded using the Coifman-Meyer Theorem \eqref{CM} to obtain
\[
\|C_1[\partial_x^{-1}f_{yy},f_y,f_y]\|_{L^1} + \|C_2[f_y,f_y,\partial_x^{-1}f_{yy}]\|_{L^1}\lesssim \|f_y\|_{L^4}^2\|\partial_x^{-1}f_{yy}\|_{L^2}.
\]

\smallskip
\emph{Proof of \eqref{est:Zk}.} The estimate \eqref{est:Zk} then follows from the conservation of mass and estimates \eqref{est:D4u}, \eqref{est:Lyu} and Gronwall's inequality.

\smallskip
\emph{Estimates for \(Ju\): Short times.} For short times (\(0<t<1\)) we take \(w = Ju + tuu_x\) and calculate
\[
\mc Lw = (uw)_x - t\partial_x[\Til^{-1}\partial_x,u]u_x - t\partial_x[(1 + \Til^{-2})\partial_x^2,u]u_x.
\]

We observe that \(\Til^{-1}\partial_x\) is a smooth Fourier multiplier with principle symbol homogeneous of order \(1\) and that \((1 + \Til^{-2})\partial_x^2\) is a smooth Fourier multiplier with Schwartz symbol. Standard commutator estimates (see for example~\cite{MR1766415}) then yield the bounds
\begin{align*}
\|\partial_x[\Til^{-1}\partial_x,u]u_x\|_{L^2} &\lesssim \|u_x\|_{L^\infty}\left(\|u_{xx}\|_{L^2} + \|u\|_{L^2}\right),\\
\|\partial_x[(1 + \Til^{-2})\partial_x^2,u]u_x\|_{L^2} &\lesssim \|u_x\|_{L^\infty}\|u\|_{L^2}.
\end{align*}

Integrating by parts in the first term we then obtain the estimate
\begin{equation}\label{est:JShort}
\frac{d}{dt}\|w\|_{L^2}^2\lesssim \|u_x\|_{L^\infty}\|w\|_{L^2}\left(\|w\|_{L^2} + \|u_{xx}\|_{L^2} + \|u\|_{L^2}\right).
\end{equation}

\emph{Estimates for \(Ju\): Long times.} 
For times \(t\geq1\) we write
\[
J = \mb S - \frac12 L_y\partial_y - 2t\mc L,
\]
where the operator
\[
\mb S = 2t\partial_t + x\partial_x + \frac32 y\partial_y - t\partial_x^3(1 + \Til^{-2}) + t\partial_x,
\]
satisfies
\[
[\mc L,\mb S] = 2\mc L.
\]
We note that in the limit \(h = \infty\) we obtain the operator \(\mb S_\infty = 2t\partial_t + x\partial_x + \frac32y\partial_y\), which is the generator of the scaling symmetry \eqref{Scaling}.

As a consequence, we define
\[
w = \mb Su - \frac12 L_y\partial_y u + \frac12u = Ju + \frac12u - 2tuu_x,
\]
and calculate
\[
\mc Lw = (uw)_x - \frac1{4t}\left(u_xL_y^2\partial_xu - (L_y\partial_xu)^2\right) - t\partial_x[\partial_x^2(1 + \Til^{-2}),u]u_x
\]

In order to obtain long time bounds for the commutator term we will take advantage of the non-resonance and use a normal form transformation to upgrade it to a cubic term. We start by using the Fourier transform to write
\[
\mc F\left[\partial_x[\partial_x^2(1 + \Til^{-2}),u]v_x\right] = \int_{\K_1 + \K_2 = \K}k(\xi_1,\xi_2)\hat u(\K_1)\hat v(\K_2)\,d\K,
\]
where the symbol
\[
k(\xi_1,\xi_2) = (\xi_1 + \xi_2)\xi_2\left((\xi_1 + \xi_2)^2\cosech^2(\xi_1 + \xi_2) - \xi_2^2\cosech^2(\xi_2) \right).
\]
Next we symmetrize to obtain
\begin{align*}
k_{\mr{sym}}(\xi_1,\xi_2) &= \frac12k(\xi_1,\xi_2) + \frac12k(\xi_2,\xi_1)\\
&= \frac12(\xi_1 + \xi_2)^4\cosech^2(\xi_1 + \xi_2) - \frac12(\xi_1 + \xi_2)\xi_1^3\cosech^2\xi_1 - \frac12(\xi_1 + \xi_2)\xi_2^3\cosech^2\xi_2.
\end{align*}
We then construct a symmetric bilinear Fourier multiplier \(B[u,v]\) with symbol
\[
b(\K_1,\K_2) = \frac{k_{\mr{sym}}(\xi_1,\xi_2)}{\Omega(\K_1,\K_2)},
\]
where the resonance function \(\Omega\) is defined as in \eqref{ResonanceFunction}. The symbol \(b\) may be readily seen to be rapidly decaying at high frequencies. However, due to the commutator structure of \(k\) it also has additional smallness at low frequencies that will allow us to obtain bounds in terms of pointwise norms of \(u_x\) rather than \(u\). We remark that in the infinite depth case, we replace \(1 + \Til^{-2}\) by \(1 + \Hil^{-2}\) and hence this term vanishes, i.e. \(k\equiv0\).

By construction we have
\[
\mc L B[u,u] = \partial_x[\partial_x^2(1 + \Til^{-2}),u]u_x + 2B[u,uu_x],
\]
so taking \(q = w + tB[u,u]\) we obtain
\[
\mc Lq = (uq)_x - \frac1{4t}\left(u_xL_y^2\partial_xu - (L_y\partial_xu)^2\right) + (1 - tu_x)B[u,u] + 2t \left(B[u,uu_x] - uB[u,u_x]\right).
\]
We claim that
\begin{align}
\|B[u,u]\|_{L^2} &\lesssim \|u_x\|_{L^\infty}\|u\|_{L^2},\label{GoodNFBounds1}\\
\|B[u,uu_x] - uB[u,u_x]\|_{L^2} &\lesssim \|u_x\|_{L^\infty}^2\|u\|_{L^2},\label{GoodNFBounds2}
\end{align}
so applying these bounds along with the \(L^4\) estimate \eqref{L4Bound} and integrating by parts in the first term we obtain the estimate
\[
\frac d{dt}\|q\|_{L^2}^2\lesssim \|u_x\|_{L^\infty}\|q\|_{L^2}^2 + \|u_x\|_{L^\infty}(1 + t\|u_x\|_{L^\infty})\|q\|_{L^2}\left(\|u\|_{L^2} + \|L_y^2\partial_xu\|_{L^2}\right),
\]
which suffices to complete the proof of \eqref{est:NRG}.

It remains to prove the estimates \eqref{GoodNFBounds1}, \eqref{GoodNFBounds2}. By considering the asymptotic behavior of \(\coth\xi\) we obtain the following asymptotic behavior in the low-high and high-high regimes,
\[
|k_{\mr{sym}}(\xi_1,\xi_2)|\lesssim \begin{cases}|\xi_1||\xi_2|^3\la \xi_2\ra^{-2},&\quad |\xi_1|\ll|\xi_2|,\vspace{0.1cm}\\|\xi_1 + \xi_2|^2|\xi_2|^2\la \xi_2\ra^{-2},&\quad |\xi_1 + \xi_2|\ll|\xi_2|.\end{cases}
\]
Combining these bounds with the estimate \eqref{FDResonanceLB} for the resonance function we obtain a (crude) bound for the symbol \(b\) of the bilinear operator \(B\),
\[
|b(\K_1,\K_2)|\lesssim \begin{cases}|\xi_2|\la\xi_2\ra^{-1},&\quad |\xi_1|\ll|\xi_2|,\vspace{0.1cm}\\|\xi_1 + \xi_2|\la\xi_2\ra^{-1},&\quad |\xi_1 + \xi_2|\ll|\xi_2|.\end{cases}
\]

Next we decompose the operator \(B\) using the Littlewood-Paley trichotomy as
\[
B[u,v] = Q_1[u,v_x] + Q_1[v,u_x] + \partial_xQ_2[u,v],
\]
where we define the bilinear forms
\[
Q_1[u,v] = \sum\limits_N B[u_{\ll N},\partial_x^{-1}v_N],\qquad Q_2[u,v] =\sum\limits_{N_1\sim N_2} \partial_x^{-1}B[u_{N_1},v_{N_2}].
\]
From the above estimates for \(b\) we see that the corresponding symbols \(q_1,q_2\) are bounded and applying similar estimates for the derivatives we obtain \(q_1,q_2\in\symb\). We may then apply the Coifman-Meyer Theorem \eqref{CM} to obtain the estimates
\[
\|Q_1[u,u_x]\|_{L^2}\lesssim \|u_x\|_{L^\infty}\|u\|_{L^2},\qquad \|\partial_xQ_2[u,u]\|_{L^2}\lesssim \|u_x\|_{L^\infty}\|u\|_{L^2},
\]
which suffice to complete the proof of \eqref{GoodNFBounds1}.

The proof of the estimate \eqref{GoodNFBounds2} is similar, taking advantage of the commutator structure. We first write that the difference
\[
B[u,uu_x] - uB[u,u_x] = C[u,u,u_x],
\]
where the operator \(C\) has symbol,
\[
C(\K_1,\K_2,\K_3) = b(\K_1,\K_2 + \K_3) - b(\K_1,\K_3).
\]
Next we decompose \(C\) according to frequency balance of the last two terms,
\[
C[u,u,u_x] = R_1[u,u_x,u_x] + R_2[u,u_x,u_x],
\]
where we define the trilinear forms,
\[
R_1[u,v,w] = \sum\limits_NC[u,\partial_x^{-1}v_N,w_{\leq N}],\qquad R_2[u,v,w] = \sum\limits_NC[u,\partial_x^{-1}v_N,w_{> N}].
\]
For the first term we may use the above estimates for the symbol \(b\) to see that \(R_1\) has symbol \(r_1\in\symb \). We then apply the Coifman-Meyer Theorem \eqref{CM} to obtain the estimate
\[
\|R_1[u,u_x,u_x]\|_{L^2}\lesssim \|u_x\|_{L^\infty}^2\|u\|_{L^2}.
\]
For the second term we instead use the commutator structure of \(C\), writing
\[
C(\K_1,\K_2,\K_3) = \int_0^1 \nabla_{\K_2}b(\K_1,h\K_2 + \K_3)\cdot\K_2\,dh,
\]
and using similar computations to above in the region \(|\xi_2|\ll|\xi_3|\) we have the estimate
\[
|C(\K_1,\K_2,\K_3)|\lesssim |\xi_2|.
\]
Applying similar estimates for the derivatives we may show that \(r_2\in\symb\) and once again we apply the Coifman-Meyer Theorem \eqref{CM} to obtain the estimate
\[
\|R_2[u,u_x,u_x]\|_{L^2}\lesssim \|u_x\|_{L^\infty}^2\|u\|_{L^2},
\]
which completes the proof of \eqref{GoodNFBounds2}.
\end{proof}
\medskip
The proof of Theorem~\ref{thm:LWP} now follows from a standard application of the energy method using the a priori estimate \eqref{prop:NRG} and the following Sobolev estimate:
\medskip
\begin{lem}\label{lem:ShortTimes}
For times \(t>0\) we have the estimate
\begin{equation}\label{ShortTime}
\|u\|_{L^\infty} + \|u_x\|_{L^\infty}\lesssim t^{-\frac12}\|u\|_{Z^3}
\end{equation}
\end{lem}
\begin{proof}
We start by decomposing \(u\) with respect to \(x\)-frequency as
\[
u = u_{<1} + \sum\limits_{N\geq1}u_N.
\]
Writing
\[
f(t,x,y) = u_{<1}(t,x - \frac{1}{4t}y^2,y),
\]
we may apply the Sobolev estimate \eqref{est:Sobolev} to obtain
\[
\|u_{<1}\|_{L^\infty}\lesssim t^{-\frac12}\|u\|_{Z^0},\qquad \|\partial_xu_{<1}\|_{L^\infty}\lesssim t^{-\frac12}\|u\|_{Z^0}.
\]
Replacing \(u_{<1}\) by \(u_N\) for \(N\in 2^{\delta\Z}\) we obtain a similar bound,
\[
\|u_N\|_{L^\infty}\lesssim t^{-\frac12}N^{-\frac32}\|u_N\|_{Z^3},\qquad \|\partial_xu_N\|_{L^\infty}\lesssim t^{-\frac12}N^{-\frac12}\|u_N\|_{Z^3}.
\]
Summing over \(N\geq1\) we obtain the estimate \eqref{ShortTime}.
\end{proof}
\medskip

%%%%%%%%%%%%%%%%%%%%%%%%
%%% POINTWISE BOUNDS %%%
%%%%%%%%%%%%%%%%%%%%%%%%

\section{Pointwise bounds}\label{sect:KS}
In this section we prove that the energy estimates for solutions proved in Section~\ref{sect:AP} lead to corresponding pointwise bounds. In particular, we have the following result:
\medskip
\begin{prop}\label{prop:BasicPointwise}
For \(t>0\) we have the estimate
\begin{equation}\label{Starter410}
\|u_x\|_{L^\infty}\lesssim |t|^{-\frac12}\la t\ra^{-\frac12}\|u\|_X.
\end{equation}
\end{prop}
\medskip

\begin{rem}
By combining the a priori estimates of Proposition~\ref{prop:NRG} with Proposition~\ref{prop:BasicPointwise} and a standard bootstrap argument, we obtain the local well-posedness of \eqref{eqn:cc-fd} on \emph{almost} global timescales \(T \approx e^{C \epsilon^{-1}}\).
\end{rem}
\medskip

For times \(0<t< 1\), Proposition~\ref{prop:BasicPointwise} is corollary of the estimate \eqref{ShortTime}. As a consequence, it suffices to consider times \(t\geq 1\). Here we will prove a slightly more involved result that we will subsequently use to upgrade the almost global existence to global existence via a bootstrap argument.

We recall that linear solutions initially localized near the origin in space will propagate along the rays \(\Gamma_{\mb v}\) of the Hamiltonian flow. In particular, if the solution is localized at \(x\)-frequency \(N\in 2^{\delta \Z}\) then at time \(t\) it should be localized in the spatial region \(\{z\approx t m(N)\}\), where we define the spatial variable
\begin{equation}\label{zVar}
z = -\left(x + \frac1{4t}y^2\right),
\end{equation}
and the non-negative symbol
\begin{equation}\label{LittleM}
m(\xi) = 2\xi\coth\xi - \xi^2\cosech^2\xi - 1.
\end{equation}

With this heuristic in mind we decompose
\[
u = u^\hyp + u^\el,
\]
where the part of the hyperbolic piece localized at frequency \(N\in 2^{\delta\Z}\) is localized in space so that
\[
t^{-1}z\in B_N^\hyp := \{v>0: v \sim m(N)\}.
\]
We note that
\[
m(\xi)\sim \frac{\xi^2}{\la\xi\ra},
\]
so due to the uncertainty principle such a localization is only possible when \(N\geq t^{-\frac13}\). As a consequence we include the low frequencies \(N<t^{-\frac13}\) (for which we may obtain improved decay) in the elliptic piece.

To make this construction rigorous, for each \(N\geq t^{-\frac13}\) we take a smooth bump function \(\chi_N^\hyp \in C^\infty_0(\R_+)\), identically \(1\) on the set \(B_N^\hyp\) and localized up to rapidly decaying tails at frequencies \(|\xi|\lesssim \frac1{m(N)}\). We then define
\[
u^\hyp = \sum\limits_{N>t^{-\frac13},\pm}u_N^{\hyp,\pm},\qquad u^\el = u - u^\hyp,
\]
where, for each \(N\geq t^{-\frac13}\),
\[
u_N^{\hyp,\pm}(t,x,y) = \chi_N^\hyp(t^{-1}z) u_N^\pm(t,x,y),\qquad u_N^{\el}(t,x,y) = \left(1 - \chi_N^\hyp(t^{-1}z)\right)u_N(t,x,y).
\]
%

%%%%%%%%%%%%%%%%%%%%%
%%% BEGIN PICTURE %%%
%%%%%%%%%%%%%%%%%%%%%

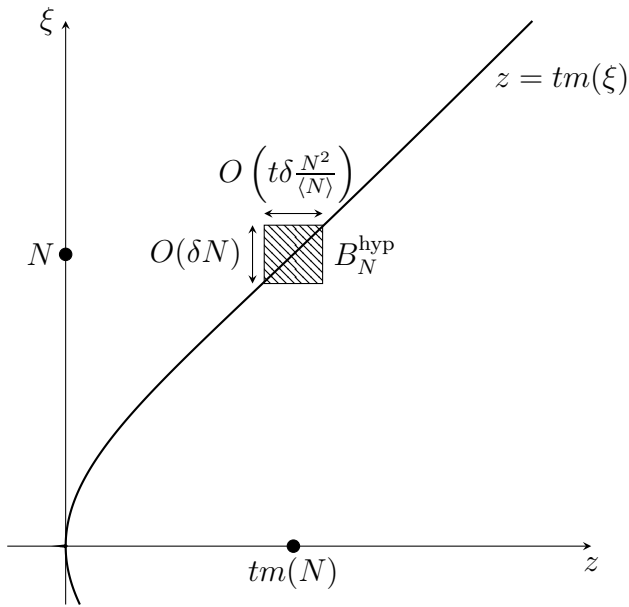
\begin{figure}\label{fig:PhaseSpace1}
\begin{tikzpicture}
\tikzset{>=stealth}
\begin{axis}[axis lines = none, no markers, xmin = -1.5, xmax=10, ymin=-1.5, ymax=9.5, axis equal image = true, scale=1.5]

% Graph
\addplot[thick, smooth, samples=1000, domain=-.5:4.5] (2*x/tanh(x) - x^2/(sinh(x))^2 - 1,2*x);
\node at (axis cs:8.5,8) {\(z = tm(\xi)\)};

% Axes
\draw[->] (axis cs:-1,0)--(axis cs:9,0) node[below] {\(z\) };
\draw[->] (axis cs:0,-1)--(axis cs:0,9) node[left] {\(\xi\)};

% Points
\filldraw (axis cs:0,5) circle (.2em) node[left] {\(N\)};
\filldraw (axis cs:3.9,0) circle (.2em) node[below] {\(tm(N)\)};

% Box
\draw[pattern=north west lines] (axis cs:3.4,4.5) rectangle (axis cs:4.4,5.5);
\node[right] at (axis cs:4.4,5) {\(B_N^\hyp\)};

% Arrows
\draw[<->] (axis cs: 3.4,5.7)--(axis cs: 4.4,5.7);
\node[above] at (axis cs: 3.8,5.7) {\(O\left(t\delta \frac{N^2}{\la N\ra}\right)\)};
\draw[<->] (axis cs: 3.2,4.5)--(axis cs: 3.2,5.5);
\node[left] at (axis cs: 3.2,5) {\(O(\delta N)\)};

\end{axis}
\end{tikzpicture}
\caption{A phase space illustration of the hyperbolic region \(B_N^\hyp\) for a fixed time \(t\geq 1\).}
\end{figure}

%%%%%%%%%%%%%%%%%%%
%%% END PICTURE %%%
%%%%%%%%%%%%%%%%%%%

If the solution \(u\) behaves like a linear wave we expect that most of its energy will be concentrated in the hyperbolic piece \(u^\hyp\) and hence the elliptic piece \(u^\el\) will be decay faster than the expected \(O(t^{-1})\) linear decay. As a consequence, we obtain the following pointwise bounds, similar to \cite[Proposition 3.1]{2014arXiv1409.4487H}:
%

%%%%%%%%%%%%%%%%%
%%% MAIN PROP %%%
%%%%%%%%%%%%%%%%%

\medskip
\begin{lem}\label{lem:FDR2} For \(t\geq 1\) and a.e. \((x,y)\in\R^2\) we have the estimates
\begin{gather}
|u^\hyp| \lesssim t^{-1}v^{-\frac38}\la v\ra^{-\frac78}\|u\|_X,\quad |u_x^\hyp|\lesssim t^{-1}v^{\frac18}\la v\ra^{-\frac38}\|u\|_X,\label{est:FDR2-HYPERBOLIC-2}\\
|u^\el| \lesssim t^{-\frac34}\la t^{\frac23}v\ra^{-\frac34}\left(1 + \log\la t^{\frac23}v\ra\right)\|u\|_X,\quad |u_x^\el|\lesssim t^{-\frac{13}{12}}\la t^{\frac23}v\ra^{-\frac14}\la t^{\frac12}v\ra^{\frac14}\la t^{-\frac12}v\ra^{-\frac5{12}}\|u\|_X\label{est:FDR2-ELLIPTIC-2}.
\end{gather}
\end{lem}
\medskip

\begin{rem}
The unusual scaling of the estimate \eqref{est:FDR2-ELLIPTIC-2} is a consequence of the fact that due to the weaker dispersion we must use additional derivatives to control the high \(x\)-frequencies rather than just the vector fields. This breaks the natural scaling of the other estimates.
\end{rem}
\medskip
\begin{rem}
In the infinite depth case we have \(m(\xi) = 2|\xi|\) and hence the high frequency threshold is \(N\geq t^{-\frac12}\) rather than \(N\geq t^{-\frac13}\). The slightly different form of the function \(m\) leads to minor adjustments to the numerology of Lemma~\ref{lem:FDR2}:
\medskip
\begin{lem} If \(h = \infty\) then for \(t\geq 1\) and a.e. \((x,y)\in\R^2\) we have the estimates
\begin{gather}
|u^\hyp| \lesssim t^{-1}v^{-\frac14}\la v\ra^{-1}\|u\|_X,\quad |u_x^\hyp|\lesssim t^{-1}v^{\frac34}\la v\ra^{-1}\|u\|_X,\\
|u^\el| \lesssim t^{-\frac78}\la t^{\frac12}v\ra^{-\frac34}\left(1 + \log\la t^{\frac12}v\ra\right)\|u\|_X,\quad |u_x^\el|\lesssim t^{-\frac98}\la t^{-\frac12}v\ra^{-\frac5{12}}\|u\|_X.
\end{gather}
\end{lem}
\end{rem}
\medskip

In order to prove Lemma~\ref{lem:FDR2} we require some auxiliary estimates so we delay the proof until Section~\ref{sect:ProofPointwise}.
 
%%%%%%%%%%%%%%%%%%%%%%%%%%%%
%%% THE EIKONAL EQUATION %%%
%%%%%%%%%%%%%%%%%%%%%%%%%%%%

\subsection{The solution to the eikonal equation}
In this section we construct the solution \(\phi\) to the eikonal equation
\begin{equation}\label{eikonal}
\ell(\phi_t,\phi_x,\phi_y) = 0,
\end{equation}
where \(\ell\) is the symbol of the linear operator \(\mc L\), given by
\[
\ell(\tau,\xi,\eta) = \tau - \xi^2\coth\xi + \xi + \xi^{-1}\eta^2.
\]
If we make the ansatz that for \(z>0\),
\[
\phi(t,x,y) = t\Phi(t^{-1}z),
\]
then we obtain an ODE for \(\Phi = \Phi(v)\),
\begin{equation}\label{EikonalODE}
\Phi - v\Phi' + (\Phi')^2\coth(\Phi') - \Phi' = 0.
\end{equation}
Differentiating we obtain
\[
(m(\Phi') - v)\Phi'' = 0,
\]
and hence
\[
\Phi'(v) = m^{-1}(v),
\]
where the positive inverse \(m^{-1}(v)>0\) may be defined using the Inverse Function Theorem for \(v>0\). As a consequence we have the following lemma:

\medskip
\begin{lem}
For \(z>0\) the solution to the eikonal equation \eqref{eikonal} is given by
\begin{equation}\label{PHASE}
\phi(t,x,y) = t\Phi(t^{-1}z),
\end{equation}
where
\[
\Phi(v) = (m^{-1}(v))^2\coth(m^{-1}(v)) - (1 + v)m^{-1}(v).
\]
\end{lem}
\medskip

We finish this section by noting that we have the estimates
\[
0< m(\xi)\sim \frac{\xi^2}{\la \xi\ra},
\]
and 
\[ 
0<m^{-1}(v)\sim \la v\ra^{\frac12}v^{\frac12}.
\]
In particular, we may show that the solution to the eikonal equation satisfies the estimate
\[
|\Phi(v)|\sim \la v\ra^{\frac12}v^{\frac32}.
\]

\medskip
\begin{rem}
In the infinite depth case we have \(\Phi(v) = -\frac14 v^2\) and hence the solution to the eikonal equation is simply
\[
\phi = -\frac{z^2}{4t}.
\]
\end{rem}
\medskip

%%%%%%%%%%%%%%%%%%%%%
%%% BEGIN PICTURE %%%
%%%%%%%%%%%%%%%%%%%%%

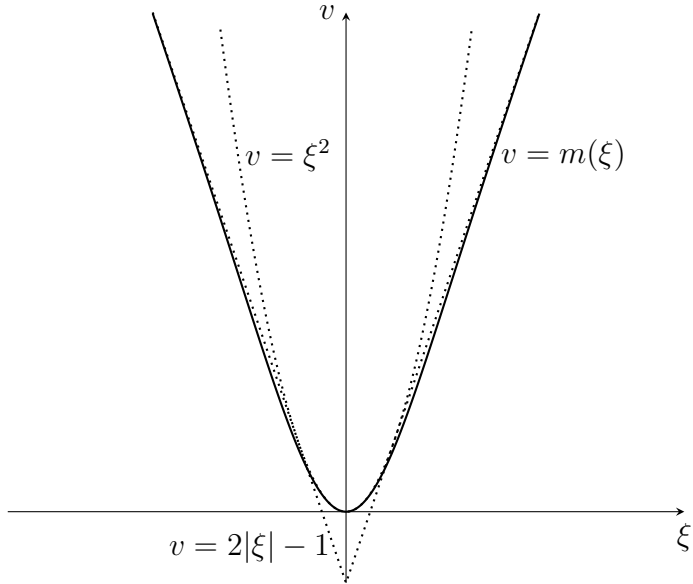
\begin{figure}
\begin{tikzpicture}
\tikzset{>=stealth}
\begin{axis}[axis lines = none, no markers, xmin = -8, xmax=8, ymin=-1.5, ymax=7.5, scale=1.5]

% Graph
\addplot[thick, smooth, samples=1000,domain = -4:4] {2*x/tanh(x) - x^2/(sinh(x))^2 - 1};
\node at (axis cs:4.5,5) {\(v = m(\xi)\)};

\addplot[thick, smooth, samples=1000,dotted,domain = -2.6:2.6] {x^2};
\addplot[thick, smooth, samples=1000,dotted,domain = -4:4] {2*abs(x) - 1};
\node at (axis cs:-1.2,5) {\(v = \xi^2\)};
\node at (axis cs:-2,-.5) {\(v = 2|\xi|-1\)};

% Axes
\draw[->] (axis cs:-7,0)--(axis cs:7,0) node[below] {\(\xi\) };
\draw[->] (axis cs:0,-1)--(axis cs:0,7) node[left] {\(v\) };
\end{axis}
\end{tikzpicture}
\caption{The graph of the function \(v = m(\xi)\).}
\end{figure}

%%%%%%%%%%%%%%%%%%%
%%% END PICTURE %%%
%%%%%%%%%%%%%%%%%%%

\subsection{Elliptic estimates}
In this section we establish weighted \(L^2\)-estimates for the frequency localized pieces \(u_N\). As we expect to obtain improved decay at very low frequencies \(N< t^{-\frac13}\) regardless, we restrict out attention to high frequencies \(N\geq t^{-\frac13}\).

In order to control the localization of solutions we define an operator adapted to the hyperbolic/elliptic decomposition of \(u\) by
\begin{equation}\label{Lz}
L_z = z - tm(D_x),
\end{equation}
so that the symbol of \(L_zP_N\) is elliptic away from the set \(B_N^\hyp\). We note that we may write
\[
L_z\partial_x = - \left(J + \frac{1}{4t}L_y^2\partial_x + \frac12\right),
\]
and hence for \(t\geq 1\) we have the estimate
\begin{equation}\label{est:LzBound}
\|L_z\partial_xu_N\|_{L^2}\lesssim \|u_N\|_X.
\end{equation}

In order to obtain more detailed estimates for the hyperbolic piece we observe that for a given \(z>0\) there exist two solutions to the equation \(m(\xi) = t^{-1}z\). As a consequence we construct operators adapted to each of these roots,
\[
L_z^\pm = m^{-1}(t^{-1}z) \pm i\partial_x,
\]
where \(L_z^-\) (respectively \(L_z^+\)) will be elliptic if \(u_N\) is localized to positive (respectively negative) wavenumbers. A useful observation, and indeed our main motivation for introducing these operators is that if \(\phi\) is the solution to the eikonal equation \eqref{eikonal} defined as in \eqref{PHASE} then we have
\[
\partial_x(e^{-i\phi}f) = -ie^{-i\phi}L_z^+f.
\]

Our main elliptic estimates are then the following:
%

%%%%%%%%%%%%%%%%%%%%%%%%%%
%%% ELLIPTIC ESTIMATES %%%
%%%%%%%%%%%%%%%%%%%%%%%%%%

\medskip
\begin{lem}\label{lem:Elliptic}
For \(t\geq1\) and \(N\geq t^{-\frac13}\) we have the estimates
\begin{align}
\frac N{\la N\ra}\|L_z^+ u_N^{\hyp,+}\|_{L^2} &\lesssim \frac1{tN}\|u_N\|_X ,\label{est:FDR2-Elliptic1}\\
\frac{N^2}{\la N\ra}\|u_N^\el\|_{L^2} + \|v u_N^\el\|_{L^2} &\lesssim \frac1{tN}\|u_N\|_X,\label{est:FDR2-Elliptic2}
\end{align}
where \(v = t^{-1}z\).
\end{lem}
\medskip
\begin{rem}
In the infinite depth case the operator
\[
L_z = z - 2t|D_x|.
\]
As a consequence, the natural analogues,
\(
L_z^\pm = \dfrac z{2t} \pm i\partial_x
\),
satisfy
\[
L_z^\pm P_\pm = \frac1{2t}L_z P_\pm,
\]
so we do not expect a gain of regularity for the hyperbolic piece as in \eqref{est:FDR2-Elliptic1} (see also~\cite{MR3382579}). For the elliptic piece we instead have the estimate
\[
N\|u_N^\el\|_{L^2} + \|v u_N^\el\|_{L^2} \lesssim \frac1{tN}\|u_N\|_X.
\]
\end{rem}
\medskip
\begin{proof}[Proof of Lemma~\ref{lem:Elliptic}]
As these are effectively one-dimensional estimates, we ignore the dependence upon \(y\) and treat \(t\geq1\) as a fixed parameter.

\emph{Proof of \eqref{est:FDR2-Elliptic1}.}
For high, positive wavenumbers \(\xi\approx N\geq1\) we may write the symbol \(m^{-1}(t^{-1}z) - \xi\) of \(L_z^+\) in terms of the symbol \(z - tm(\xi)\) of \(L_z\) as
\[
m^{-1}(t^{-1}z) - \xi = t^{-1}(z - tm(\xi))p(t^{-1}z,\xi),
\]
where the smooth function
\[
p(v,\xi) = \frac{m^{-1}(v) - \xi}{v - m(\xi)},
\]
is elliptic in the region \(v\sim m(\xi)\) for \(\xi\geq1\). We recall that \(u_{N,+}^\hyp\) is localized in space in the set \(B_N^\hyp\) and in frequency at positive wavenumbers \(\xi\approx N\) up to rapidly decaying tails at scale \(tN\). In particular we may harmlessly localize \(p(v,\xi)\) using cutoffs in space and frequency and then apply the product rule \eqref{ProductRule} and the elliptic estimate \eqref{OPSElliptic} to obtain
\[
\|L_{z,+}u_N^\hyp\|_{L^2}\lesssim \frac1t\left(\|L_zu_N^\hyp\|_{L^2} + \|u_N^\hyp\|_{L^2}\right) \lesssim \frac1{tN}\|u_N\|_X.
\]

For low, positive wavenumbers \(\xi\approx N\) so that \(t^{-\frac13}\leq N < 1\) we instead write the product of the symbols of \(L_z^-L_z^+\) as
\[
(m^{-1}(t^{-1}z) + \xi)(m^{-1}(t^{-1}z) - \xi) = t^{-1}(z - tm(\xi))q(t^{-1}z,\xi),
\]
where
\[
q(v,\xi) = \frac{(m^{-1}(v))^2 - \xi^2}{v - m(\xi)}\in\symb
\]
A similar estimate to above yields the bound
\[
\|L_z^-L_z^+u_{N}^{\hyp,+}\|_{L^2}\lesssim \frac1{tN}\left(\|L_z\partial_xu_N\|_{L^2} + \|u_N\|_{L^2}\right).
\]

For sufficiently smooth \(w\) we then calculate
\[
\|wf\|_{L^2}^2 + \|f_x\|_{L^2}^2 = \|L_z^-f\|_{L^2}^2 + 2\Im\int wf\cdot\bar f_x.
\]
Applying this with \(w = m^{-1}(t^{-1}z)\) and \(f = L_z^+u_{N}^{\hyp,+}\), which is localized at positive wavenumbers \(\xi \sim N\) and in space in the region \(B_N^\hyp\) up to rapidly decaying tails, we obtain the estimate
\begin{align*}
\|f\|_{L^2}^2 &\lesssim \frac1{N^2}\left(\|wf\|_{L^2}^2 + \|f_x\|_{L^2}^2\right) +\frac1{t^2N^4}\|u_N\|_{L^2}^2\\
&\lesssim \frac1{N^2} \|L_z^-f\|_{L^2}^2 + \frac1{N^2}\Im\int wf\cdot\bar f_x + \frac1{t^2N^4}\|u_N\|_{L^2}^2\\
&\lesssim \frac1{t^2N^4}\left(\|L_z\partial_xu_N\|_{L^2}^2 + \|u_N\|_{L^2}^2\right)
\end{align*}
Combining these estimates we obtain \eqref{est:FDR2-Elliptic1}.

\emph{Proof of \eqref{est:FDR2-Elliptic2}.}
We define the Fourier multiplier
\[
a(\xi) = \sqrt{m(\xi)},
\]
and take \(A = a(D_x)\) so that
\[
L_z = z - tA^2.
\]
Integrating by parts for real-valued \(f\) we obtain the identity
\begin{equation}\label{est:FDR2-IBP}
\|v f_x\|_{L^2}^2 - 2\int v|Af_x|^2\,dx + \|A^2f_x\|_{L^2}^2 = t^{-2}\|L_z\partial_xf\|_{L^2}^2,
\end{equation}
where have used that the symbol of the operator \(A[A,v]\) is skew-adjoint.

We then smoothly decompose the elliptic part of \(u_N\)
\[
u_N^\el = \chi_{\{|z|\ll tm(N)\}}u_N + \chi_{\{z\approx - tm(N)\}}u_N + \chi_{\{|z|\gg tm(N)\}}u_N,
\]
where the smooth cutoffs are assumed to have compact support and be localized in frequency near zero at the scale of uncertainty.

For the first and last piece we apply the estimate \eqref{est:FDR2-IBP} with \(f = \chi_{\{|v|\ll m(N)\}}u_N\), \(\chi_{\{|z|\gg tm(N)\}}u_N\) respectively to obtain
\[
\|v f_x\|_{L^2} + m(N)\|f_x\|_{L^2} \lesssim \frac 1t\|u_N\|_X + \sqrt{m(N)}\|\sqrt vf_x\|_{L^2},
\]
so from the spatial localization of \(f\) we obtain,
\[
\|v f_x\|_{L^2} + m(N)\|f_x\|_{L^2}\lesssim \frac 1t\|u_N^\el\|_X.
\]
For the remaining piece we use that for \(f = \chi_{\{z\approx - tm(N)\}}u_N\) the function \(Af_x\) is localized in the spatial region \(v <0 \) up to rapidly decaying tails to obtain a similar estimate,
\[
\|v f_x\|_{L^2} + m(N)\|f_x\|_{L^2}\lesssim \frac 1t\|u_N^\el\|_X.
\]

Combining these bounds with the fact that \(f\) is localized at frequencies \(\sim N\) up to rapidly decaying tails, we obtain the estimate \eqref{est:FDR2-Elliptic2}.
\end{proof}
\medskip

\begin{rem}
We obeserve that we may combine the estimate \eqref{est:FDR2-Elliptic2} with the elementary low frequency estimate
\[
\|\partial_xu_{\leq t^{-\frac13}}\|_{L^2}\lesssim t^{-\frac13}\|u\|_{L^2},
\]
to obtain the estimate for the elliptic piece
\begin{equation}\label{est:EllipticGain}
\|\partial_xu^{\el}\|_{L^2}\lesssim t^{-\frac13}\|u\|_X.
\end{equation}

Further, in the infinite depth case we have the corresponding estimate,
\begin{equation}
\|\partial_xu^{\el}\|_{L^2}\lesssim t^{-\frac12}\|u\|_X.
\end{equation}
\end{rem}
\medskip

\subsection{Proof of Lemma~\ref{lem:FDR2}}\label{sect:ProofPointwise}
We now apply the elliptic estimates of Lemma~\ref{lem:Elliptic} to prove Lemma~\ref{lem:FDR2}. As the estimates are linear in \(u\) we will assume that \(\|u\|_X = 1\). Throughout this section we use the notation \(v = t^{-1}z\).

\emph{Low frequencies.} We first consider the low frequency part \(u_{\leq t^{-\frac13}}\). We recall that at low frequency \(m(\xi)\approx \xi^2\) and hence the operator \(t^{-1}L_z = v - m(D_x) \approx v\) whenever \(v\gg t^{-\frac23}\) whereas \(t^{-1}L_z\approx \partial_x^2\) whenever \(v\ll t^{-\frac23}\).

If \(|v|< t^{-\frac23}\) we take \(f(t,x,y) = u_{< t^{-\frac13}}(t,x - \frac{1}{4t}y^2,y)\). We then apply Bernstein's inequality to obtain the bounds
\begin{gather*}
\|f\|_{L^2}\lesssim \|u_{<t^{-\frac13}}\|_{L^2}\lesssim 1,\qquad \|f_x\|_{L^2}\lesssim \|\partial_xu_{<t^{-\frac13}}\|_{L^2}\lesssim t^{-\frac13},\\ \|f_{yy}\|_{L^2}\lesssim t^{-2}\|L_y^2\partial_x^2u_{<t^{-\frac13}}\|_{L^2}\lesssim t^{-\frac73}.
\end{gather*}
Applying the Sobolev estimate as in Lemma~\ref{lem:ShortTimes} we obtain the estimate,
\[
\|u_{<t^{-\frac13}}\|_{L^\infty}\lesssim t^{-\frac34}.
\]
Essentially identical estimates to \(f_x\) yield a similar bound,
\[
\|\partial_xu_{< t^{-\frac13}}\|_{L^\infty}\lesssim t^{-\frac{13}{12}}.
\]

If instead \(|v|\geq t^{-\frac23}\) we dyadically decompose in space, taking \(\chi_M\) to localize to the spatial region \(\{|v|\sim M\}\) for each \(M>t^{-\frac23}\). We then decompose \(\chi_Mu_{\leq t^{-\frac13}}\) at the scale of the uncertainty principle as
\[
\chi_Mu_{< t^{-\frac13}} = \chi_Mu_{< (tM)^{-1}} + \sum\limits_{(tM)^{-1}\leq N< t^{-\frac13}}\chi_Mu_N.
\]
For any \(N<t^{-\frac13}\) we have the estimate
\[
\|v\partial_xu_N\|_{L^2}\lesssim t^{-1}\left(\|L_z\partial_xu\|_{L^2} + tNm(N)\|u\|_{L^2}\right) \lesssim t^{-1},
\]
where we have used that \(Nm(N)\approx N^3\lesssim t^{-1} \) whenever \(N<t^{-\frac13}\). As a consequence we obtain the bounds,
\[
\|\partial_x(\chi_Mu_{< (tM)^{-1}})\|_{L^2}\lesssim (tM)^{-1},\qquad \|\partial_x(\chi_Mu_N)\|_{L^2}\lesssim (tM)^{-1}.
\]
We then apply the Sobolev estimate \eqref{est:Sobolev} as before to
\[
f(t,x,y) = \chi_M(t^{-1}x)u_{< (tM)^{-1}}(t,x - \frac{1}{4t}y^2,y),\qquad f(t,x,y) = \chi_M(t^{-1}x)u_N(t,x - \frac{1}{4t}y^2,y),
\]
respectively to obtain the bounds,
\[
\|\chi_M u_{< (tM)^{-1}}\|_{L^\infty} + \|\chi_M u_N\|_{L^\infty}\lesssim t^{-\frac12}(tM)^{-\frac34}.
\]
Replacing \(u\) by \(u_x\) we obtain similar bounds,
\[
\|\chi_M \partial_xu_{< (tM)^{-1}}\|_{L^\infty}\lesssim t^{-\frac12}(tM)^{-\frac74},\qquad \|\chi_M \partial_xu_N\|_{L^\infty} \lesssim t^{-\frac12}(tM)^{-\frac34}N.
\]
Summing over \(N\) we obtain the estimates,
\[
|\chi_Mu_{<t^{-\frac13}}|\lesssim t^{-\frac34}(t^{\frac23}M)^{-\frac34}\log(t^{\frac23}M),\qquad |\chi_M\partial_xu_{<t^{-\frac13}}|\lesssim t^{-\frac{13}{12}}(t^{\frac23}M)^{-\frac34},
\]
where the logarithmic loss arises due to summation over frequencies \((tM)^{-1}\leq N\leq t^{-\frac13}\) (see also the corresponding bound in \cite[Proposition~3.1]{2014arXiv1409.4487H}).

Combining these bounds with the estimate in the region for \(|v|<t^{-\frac23}\) we obtain the low frequency estimates
\begin{equation}
|u_{< t^{-\frac13}}|\lesssim t^{-\frac34}\la t^{\frac23} v\ra^{-\frac34}\left(1 + \log\la t^{\frac23} v\ra\right),\qquad |\partial_xu_{< t^{-\frac13}}|\lesssim t^{-\frac{13}{12}}\la t^{\frac23} v\ra^{-\frac34}.\label{est:FDR2-LF-2}
\end{equation}

\emph{Elliptic piece.} For the high frequency part of the elliptic piece we proceed similarly to the low frequency piece. For \(N\geq t^{-\frac13}\) we apply the Sobolev estimate \eqref{est:Sobolev} and the elliptic estimate \eqref{est:FDR2-Elliptic2} to \(f(t,x,y) = u_N^\el(t,x - \frac{1}{4t}y^2,y)\) on dyadic spatial intervals (as for the low frequency piece) to obtain the pointwise bound,
\[
|u^\el_N|\lesssim t^{-\frac54}\min\{N^{-\frac32}\la N\ra^{\frac34},|v|^{-\frac34}\}.
\]
If \(|v|< t^{-\frac23}\) then we sum in \(N\) to obtain
\[
|u^\el_{\geq t^{-\frac13}}|\lesssim \sum\limits_{N\geq t^{-\frac13}}t^{-\frac54}N^{-\frac32}\la N\ra^{\frac34}\lesssim t^{-\frac34}.
\]
If \(|v|\geq t^{-\frac23}\) then we decompose the sum as
\[
|u^\el_{\geq t^{-\frac13}}|\lesssim \sum\limits_{t^{-\frac13}\leq N\leq v^{\frac12}\la v\ra^{\frac12}} t^{-\frac54}|v|^{-\frac34} +  \sum\limits_{N\geq v^{\frac12}\la v\ra^{\frac12}}t^{-\frac54}N^{-\frac32}\la N\ra^{\frac34} \lesssim t^{-\frac34}(t^{\frac23}|v|)^{-\frac34}\left(1 + \log (t^{\frac13}v^{\frac12}\la v\ra^{\frac12})\right).
\]
Combining these, we obtain the bound,
\[
|u^\el_{\geq t^{-\frac13}}|\lesssim t^{-\frac34}\la t^{\frac23}v\ra^{-\frac34}\left(1 + \log\la t^{\frac23}v\ra\right).
\]

As \(u_N^\el\) is localized at frequencies \(|\xi|\sim N\) up to rapidly decaying tails, we obtain a similar estimate for the derivative,
\[
|\partial_xu^\el_N|\lesssim t^{-\frac54}N\min\{N^{-\frac32}\la N\ra^{\frac34},|v|^{-\frac34}\},
\]
Summing over \(t^{-\frac13}\leq N<1\) we then obtain the bound,
\[
|\partial_xu^\el_{t^{-\frac13}\leq \cdot <1}|\lesssim t^{-\frac{13}{12}}\la t^{\frac23}v\ra^{-\frac14}.
\]

For \(N>1\) we may use the fact that \(\|\partial_x^4u_N\|_{L^2}\lesssim 1\) in lieu of the elliptic estimate \eqref{est:FDR2-Elliptic2} to obtain the slight modification
\[
|\partial_xu^\el_N|\lesssim t^{-\frac12}N\min\{(tN)^{-\frac34},(t|v|)^{-\frac34},N^{-\frac94}\},
\]
from which we obtain the bound
\[
|\partial_xu^\el_{>1}|\lesssim t^{-\frac98}\la t^{-\frac12}v\ra^{-\frac5{12}},
\]
completing the proof of \eqref{est:FDR2-ELLIPTIC-2}.

\emph{Hyperbolic piece.} We define the phase function \(\phi\) as in \eqref{PHASE} and observe that if we apply the Sobolev estimate \eqref{est:Sobolev} to \(f(t,x,y) = e^{-i\phi}u_N^{\hyp,+}(t,x - \frac{1}{4t}y^2,y)\) we obtain the estimate
\[
|u_N^{\hyp,+}|\lesssim t^{-\frac12}\|u_N\|_{L^2}^{\frac14}\|L_z^+u_N^{\hyp,+}\|_{L^2}^{\frac12}\|(L_y\partial_x)^2u_N\|_{L^2}^{\frac14}.
\]
We then use the elliptic estimate \eqref{est:FDR2-Elliptic1} and that \(u_N^{\hyp,+}\) is localized in the spatial region \(v\approx m(N)\) and at frequencies \(\sim N\) up to rapidly decaying tails to obtain
\[
|u_N^{\hyp,+}|\lesssim t^{-1}N^{-\frac34}\la N\ra^{-\frac12},\qquad |\partial_xu_N^{\hyp,+}|\lesssim t^{-1}N^{\frac14}\la N\ra^{-\frac12}
\]

If \(t^{-\frac13}\leq N<1\) then on the support of \(u_N^{\hyp,+}\) we have \(v\approx N^2\) so using that the supports of the \(u_N^{\hyp,+}\) are essentially disjoint we may sum to obtain
\[
|u_{t^{-\frac13}\leq \cdot<1}^{\hyp,+}|\lesssim t^{-1}v^{-\frac38},\qquad |\partial_xu_{t^{-\frac13}\leq \cdot<1}^{\hyp,+}|\lesssim t^{-1}v^{\frac18}.
\]
Similarly, if \(N\geq1\) we have \(v\approx N\) on the support of \(u_N^{\hyp,+}\) and hence
\[
|u_{\geq 1}^{\hyp,+}|\lesssim t^{-1}v^{-\frac54},\qquad |\partial_xu_{\geq 1}^{\hyp,+}|\lesssim t^{-1}v^{-\frac14}.
\]
By combining these estimates we obtain the bound \eqref{est:FDR2-HYPERBOLIC-2}, which completes the proof of Lemma~\ref{lem:FDR2}.\qed

%%%%%%%%%%%%%%%%%%%%%%%%%%%%%%
%%% TESTING BY WAVEPACKETS %%%
%%%%%%%%%%%%%%%%%%%%%%%%%%%%%%

\section{Testing by wave packets}\label{sect:TWP}
We now turn to the problem of proving the existence of global solutions to \eqref{eqn:cc-fd} using a bootstrap argument. We assume that for some \(T>1\) there exists a solution \(S(-t)u\in C([0,T];X(0))\) to \eqref{eqn:cc-fd} satisfying the bootstrap assumption,
\begin{equation}\label{BS}
\sup\limits_{t\in[0,T]}\|u_x\|_{L^\infty}\leq \mc M\epsilon t^{-\frac12}\la t\ra^{-\frac12}.
\end{equation}
Applying the a priori bound \eqref{est:NRG} we obtain the estimate
\begin{equation}\label{AP}
\|u\|_X\lesssim \epsilon \la t\ra^{C\mc M\epsilon},
\end{equation}
where the constants are independent of \(\mc M\).

From the pointwise bounds proved in Lemma~\ref{lem:FDR2}, we see that the worst decay occurs in the region \(\{z\approx t\}\). As a consequence we define the time-dependent set
\[
\Sigma_t = \{v\in\R: t^{-\frac1{12}}< v < t^{\frac1{12}}\},
\]
so that the worst behavior will occur whenever \(t^{-1}z\in \Sigma_t\). If we take \(\chi_{\Sigma_t^c}\) to be a smooth bump function supported in the complement \(\Sigma_t^c = \R^2\backslash\Sigma_t\) we may then apply the estimates of Lemma~\ref{lem:FDR2} to obtain
\begin{equation}\label{ImprovedDecay}
\|u_x\ \chi_{\Sigma_t^c}(t^{-1}z)\|_{L^\infty}\lesssim t^{-\frac{97}{96}}\|u\|_X.
\end{equation}
The additional \(t\)-decay in the estimate \eqref{ImprovedDecay} leads to an improvement of \eqref{BS} for sufficiently large times. As a consequence it remains to consider improved pointwise bounds for \(u_x\) in the region \(\Sigma_t\).

%%%%%%%%%%%%%%%%%%%%%
%%% BEGIN PICTURE %%%
%%%%%%%%%%%%%%%%%%%%%

\begin{figure}
\begin{tikzpicture}
\tikzset{>=stealth}
\begin{axis}[axis lines = none, no markers, xmin = -2, xmax=13, ymin=-1.5, ymax=15.5, scale=1.5,stack plots = y]

% Graph
\addplot+[thick, smooth,domain = 1:11.9,color=black] {x^((13)/(12))};
\addplot+[thick, smooth,domain = 1:11.9,color=black] {x^((11)/(12)) - x^((13)/(12))};

\addplot[pattern = north west lines, smooth,domain = 1:11.9] {- x^((11)/(12)) + x^((13)/(12))} \closedcycle;

\node at (axis cs:9,6) {\(\{\frac zt\in \Sigma_t\}\)};

% Axes
\draw[->] (axis cs:-1,0)--(axis cs:11.9,0) node[below] {\(z\) };
\draw[->] (axis cs:0,-1)--(axis cs:0,15) node[left] {\(t\) };

\filldraw (axis cs: 0,1) circle (.15em) node[left] {\(1\)};

\end{axis}
\end{tikzpicture}
\caption{The region \(\{\frac zt\in \Sigma_t\}\).}
\end{figure}
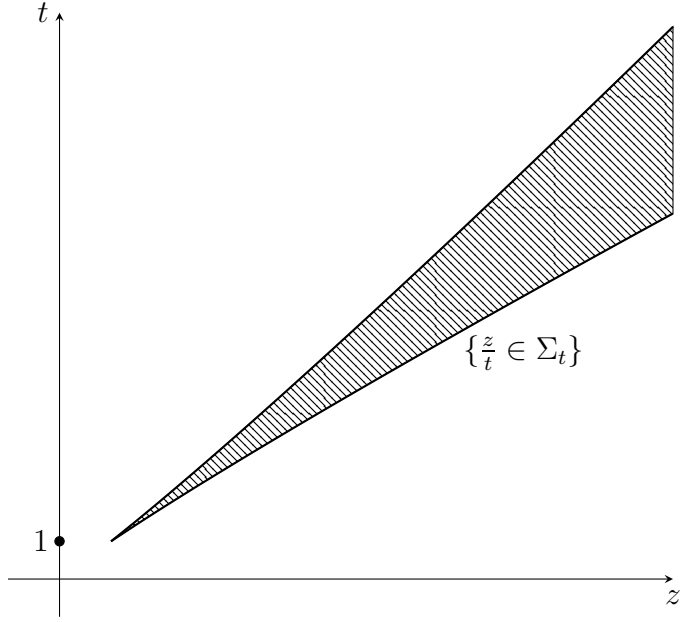

%%%%%%%%%%%%%%%%%%%
%%% END PICTURE %%%
%%%%%%%%%%%%%%%%%%%

\subsection{Construction of the wave packets} Given a time \(t\geq 1\) and a velocity \(\mb v = (v_x,v_y)\in\R^2\) such that
\[
v = - \left(v_x + \frac 14 v_y^2\right) \in \Sigma_t,
\]
we construct a wave packet adapted to the associated ray \(\Gamma_{\mb v} = \{(x,y) = t\mb v\}\) of the Hamiltonian flow by
\[
\Psi_{\mb v}(t,x,y) = \partial_x\left(\frac1{i\partial_x\phi}e^{i\phi}\chi\left(\lambda_z(z - tv),\lambda_y(y - tv_y)\right)\right),
\]
where \(\chi\in C^\infty_0(\R^2)\) is a smooth, non-negative, real-valued, compactly supported function, localized near \(0\) in space and frequency at scale \(\lesssim 1\), the phase \(\phi\) is defined as in \eqref{PHASE} and the scales
\[
\lambda_z = t^{-\frac12}v^{-\frac14}\la v\ra^{\frac14},\qquad \lambda_y = t^{-\frac12}v^{\frac14}\la v\ra^{\frac14}.
\]
For simplicity we normalize \(\int_{\R^2} \chi(z,y)\,dzdy = 1\). We also note that by a slight abuse of notation we consider \(v\) to be a fixed parameter (independent of \(t,z\)) in this section.

\medskip
\begin{rem}
If our initial data is localized near the origin in space and at frequency \(\K_0 = (\xi_0,\eta_0)\) then the corresponding linear solution will be spatially localized on the ray \(\Gamma_{\mb v} = \{z = t\mb v\}\) of the Hamiltonian flow, where the group velocity
\[
\mb v = - \nabla\omega(\K_0) = (- m(\xi_0) - \xi_0^{-2}\eta_0^2,2\xi_0^{-1}\eta_0).
\]
We note that the frequency may then be written in terms of the velocity as
\[
\K_0 = \left(m^{-1}(v),\frac12v_y  m^{-1}(v)\right).
\]

If a linear solution is localized near the ray \(\Gamma_{\mb v}\) in space at scale \(\lambda_z^{-1}\) in \(z\) and \(\lambda_y^{-1}\) in \(y\) then from the uncertainty principle, the Fourier transform may be localized at most such that
\[
\left|\xi - \xi_0\right|\lesssim\lambda_z,\qquad \left|(\eta - \eta_0) - \frac{\eta_0}{\xi_0}(\xi - \xi_0)\right|\lesssim \lambda_y.
\]
In order for a function to be coherent on timescales \(\approx T\) we require that \(\omega(\K)\) may be well-approximated by its linearization, with errors of size \(\ll T^{-1}\). Computing the Taylor expansion of the dispersion relation \(\omega(\K)\) at frequency \(\K = \K_0\) we obtain
\[
\omega(\K) = \omega(\K_0) - \mb v\cdot(\K - \K_0) + \frac12(\K - \K_0)\cdot\nabla^2\omega(\K_0) (\K - \K_0) + \dots.
\]
With the above localization we calculate
\[
\left|(\K - \K_0)\cdot\nabla^2\omega(\K_0) (\K - \K_0)\right|\lesssim m'(\xi_0) \lambda_z^2 + \xi_0^{-1} \lambda_y^2.
\]
Thus we require,
\[
T\lambda_z^2\lesssim \frac1{m'(\xi_0)}\sim \frac{\la v\ra^{\frac12}}{v^{\frac12}},\qquad T\lambda_y^2\lesssim \xi_0\sim v^{\frac12}\la v\ra^{\frac12},
\]
which motivates the choice of scales.
\end{rem}
\medskip
\begin{rem}
In the infinite depth case \(h = \infty\) essentially identical reasoning yields the scales
\[
\lambda_z = t^{-\frac12},\qquad \lambda_y = t^{-\frac12}v^{\frac12}.
\]
\end{rem}
\medskip
%

%%%%%%%%%%%%%%%%%%%%%
%%% BEGIN PICTURE %%%
%%%%%%%%%%%%%%%%%%%%%

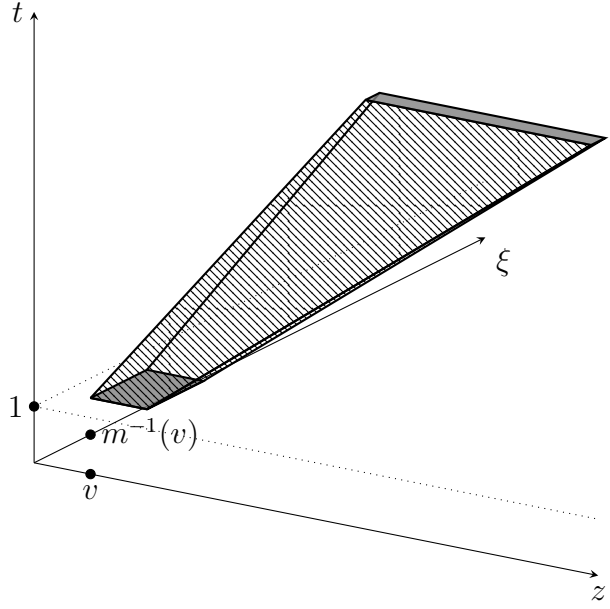
\begin{figure}
\begin{tikzpicture}[scale=0.75]
\tikzset{>=stealth}

% Axes
\draw[->] (0,0)--(0,8) node[left] {\(t\) };
\draw[->] (0,0)--(8,4) node[below right] {\(\xi\)};
\draw[->] (0,0)--(10,-2) node[below] {\(z\)};

% Points
\filldraw (0,1) circle (.2em) node[left] {\(1\)};
\filldraw (1,.5) circle (.2em) node[right] {\(m^{-1}(v)\)};
\filldraw (1,-.2) circle (.2em) node[below] {\(v\)};

% t=1 lines
\draw[dotted] (0,1)--(8,5);
\draw[dotted] (0,1)--(10,-1);

% Center line
% \draw (2,1.3)--(8,6.1);
% \draw[dashed] (1,.5)--(2,1.3);

% Sides of box

\filldraw[thick,fill=gray,fill opacity=.8] (1,1.15)--(2,.95)--(3,1.45)--(2,1.65)--(1,1.15);
\draw[thick] (2,1.65)--(6.125,6.5625);

\draw[thick,pattern = north west lines] (2,.95)--(9.875,5.6375)--(10.125,5.7625)--(3,1.45)--(2,.95);
\draw[thick,pattern = north west lines] (1,1.15)--(5.875,6.4375)--(9.875,5.6375)--(2,.95)--(1,1.15);

% Final box
\filldraw[thick,fill=gray,fill opacity=.8] (5.875,6.4375)--(9.875,5.6375)--(10.125,5.7625)--(6.125,6.5625)--(5.875,6.4375);

\end{tikzpicture}
\caption{An illustration of the localization of the wave packet \(\Psi_{\mb v}\) in \((z,\xi)\)-phase space.}
\end{figure}

%%%%%%%%%%%%%%%%%%%
%%% END PICTURE %%%
%%%%%%%%%%%%%%%%%%%
%
In order to clarify these heuristics we first make the following definition:
\medskip
\begin{defn}
Given a fixed time \(t\geq1\), a velocity \(\mb v\in \R^2\) so that \(v\in \Sigma_t\) and a (possibly \(t,v\)-dependent) constant \(\Upsilon>0\), we say \(f\in WP(t,\mb v,\Upsilon)\) if \(f\in C^\infty_0(\R^2)\) is supported in the set \(\{\frac zt\in \tilde\Sigma_t\}\), where \(\tilde \Sigma_t\) is a slight dilation of \(\Sigma_t\), and for all \(\alpha,\beta,\mu,\nu\geq 0\) we have the estimate,
\begin{equation}
\left|(z - tv)^\mu (y - tv_y)^\nu(\partial_x - i\partial_x\phi)^\alpha (\partial_xL_y)^\beta f\right|\lesssim_{\alpha,\beta,\mu,\nu}\lambda_z^{\alpha - \mu}\lambda_y^{\beta - \nu}\Upsilon.
\end{equation}
\end{defn}
\medskip
\noindent We note that if \(f\in WP(t,\mb v,\Upsilon)\) then it is localized in space near the ray \(\Gamma_{\mb v}\) and in frequency near the corresponding frequency \(\K = \left(m^{-1}(v),\frac12v_y  m^{-1}(v)\right)\) at the scale of uncertainty. Using this definition we may clarify the structure of the wave packet \(\Psi_{\mb v}\):
\medskip
\begin{lem}\label{lem:Microlocal}
For all times \(t\gg1\) sufficiently large and all \(\mb v\in \R^2\) be chosen such that \(v\in\Sigma_t\), the associated wave packet \(\Psi_{\mb v}\in WP(t,\mb v,1)\) and \(\mc L\Psi_{\mb v}\in WP(t,\mb v,t^{-1})\).

Further, writing \(\chi = \chi(\lambda_z(z - tv),\lambda_y(y - tv_y))\), we have the decomposition
\begin{equation}\label{LinearForm}
\begin{aligned}
e^{-i\phi}\mc L\Psi_{\mb v} &= - \partial_x\left(\frac1{2t}(z - tv)\chi\right) - \partial_xL_y\left(\frac1{4t^2}(y - tv_y)\chi\right) + \partial_x\left(\frac12im'(\partial_x\phi)\partial_x\chi\right)\\
&\quad + (\partial_xL_y)^2\left(\frac1{i4t^2\partial_x\phi}\chi\right) + \err,
\end{aligned}
\end{equation}
where the error term \(\err\in WP(t,\mb v,t^{-\frac32}v^{-\frac34}\la v\ra^{-\frac14})\).
\end{lem}
\begin{proof}
We may write the wave packet in the form
\begin{equation}\label{LeadingOrderDecomp}
e^{-i\phi}\Psi_{\mb v} = \chi\left(\lambda_z(z - tv),\lambda_y(y - tv_y)\right) + \frac{\lambda_z}{i\partial_x\phi}\chi_z\left(\lambda_z(z - tv),\lambda_y(y - tv_y)\right).
\end{equation}
We recall that \(\partial_x\phi = m^{-1}(t^{-1}z)\) and a simple application of the Inverse Function Theorem yields the estimates
\begin{gather*}
m^{-1}(v)\sim v^{\frac12}\la v\ra^{\frac12},\qquad |\partial_vm^{-1}(v)|\lesssim v^{-\frac12}\la v\ra^{\frac12},\\
|\partial_v^\alpha m^{-1}(v)|\lesssim_k v^{-(\frac12 + \alpha)}\la v\ra^{-k},\quad \alpha\geq 2.
\end{gather*}
We then recall that if \((x,y)\in \supp\Psi_{\mb v}\) then \(|z - tv|\lesssim \lambda_z^{-1} \lesssim t^{\frac12}v^{-\frac14}\la v\ra^{\frac14}\) so provided \(t\gg1\) we have
\[
m^{-1}(t^{-1}z)\sim m^{-1}(v)
\]
As a consequence we may differentiate to obtain
\[
\left|\partial_x^\alpha\left(\frac{\lambda_z}{i\partial_x\phi}\right)\right|\lesssim_\alpha t^{-(\frac12+\alpha)} v^{-(\frac34 + \alpha)}\la v\ra^{-\frac14}\ll \lambda_z^\alpha,
\]
whenever \(\alpha\geq 0\), \(t\gg1\), \((x,y)\in \supp\Psi_{\mb v}\) and \(v\in \Sigma_t\). Differentiating \eqref{LeadingOrderDecomp} with respect to \(\partial_x\), \(\partial_xL_y\) and using the fact that \(\chi\) is compactly supported near \(0\), we obtain \(\Psi_{\mb v}\in WP(t,\mb v,1)\).

In order to calculate \(\mc L\Psi_{\mb v}\), we first define the Fourier multiplier
\[
M(\xi) = \xi^2\coth\xi - \xi,
\]
so that \(M(D_x) = i(\partial_x^2\Til^{-1} - \partial_x)\) and \(m = \partial_\xi M\). Using the Taylor expansion of the symbol about the point \(\xi = \partial_x\phi\) we obtain
\begin{equation}\label{NonlocalExpansion}
\begin{aligned}
M(D_x)\Psi_{\mb v} &= M(\partial_x\phi)\Psi_{\mb v} + m(\partial_x\phi)\left(D_x - \partial_x\phi\right)\Psi_{\mb v} + \frac12m'(\partial_x\phi)\left(D_x - \partial_x\phi\right)^2\Psi_{\mb v}\\
&\quad + \frac i{2t}\Psi_{\mb v} + \err_0,
\end{aligned}
\end{equation}
where we have used that \(\partial_x\phi = m^{-1}(t^{-1}z)\) so \(m'(\partial_x\phi)\partial_x^2\phi = -\frac1t\), and the error term \(\err_0\) may be written using the Fourier transform as
\begin{align*}
\err_0 &= \frac1{2\pi}\int \left(\frac12\int_0^1m''\left(\partial_x\phi + \tau(\xi - \partial_x\phi)\right)(1 - \tau)^2(\xi - \partial_x\phi)^3\,d\tau\right)\hat\Psi_{\mb v}(t,\xi,\eta)e^{i(\xi x + \eta y)}\,d\xi d\eta.
\end{align*}
Using the facts that \(\Psi_{\mb v}\in WP(t,\mb v,1)\) and \(|m''(\xi)|\lesssim_k \la \xi\ra^{-k}\) it is quickly observed that the error term satisfies \(\err_0\in WP(t,\mb v,t^{-\frac32}v^{-\frac34}\la v\ra^{-k})\).

As a consequence it remains to consider the approximate linear operator,
\[
\tilde{\mc L} = \partial_t - iM(\partial_x\phi) - im(\partial_x\phi)\left(D_x - \partial_x\phi\right) - \frac12im'(\partial_x\phi)\left(D_x^2 - \partial_x\phi\right)^2 + \frac1{2t} + \partial_x^{-1}\partial_y^2,
\]
which satisfies
\(
\mc L\Psi_{\mb v} = \tilde{\mc L}\Psi_{\mb v} + \err_0.
\)
Next we change variables \((t,x,y)\mapsto (t,z,y)\) so that with \(\phi = \phi(t,z) = t\Phi(t^{-1}z)\) we may write the approximate linear operator \(\tilde{\mc L}\) as
\begin{align*}
\mc{\tilde L} = \left(\partial_t - i\partial_t\phi\right) + \frac zt\left(\partial_z - i\partial_z\phi\right) + \frac12m'(-\partial_z\phi)\left(\partial_z - i\partial_z\phi\right)^2 + \frac1t - \partial_z^{-1}\partial_y^2 + \frac yt\partial_y.
\end{align*}
As a consequence, for a function \(f = f(t,z,y)\) we obtain
\begin{align*}
e^{-i\phi}\tilde{\mc L}\partial_z\left(\frac{e^{i\phi}}{i\partial_z\phi}f\right) &= \left(\partial_t + \frac zt\partial_z + \frac12m'(-\partial_z\phi)\partial_z^2 + \frac1t + \frac yt\partial_y\right)\left(f + \frac1{i\partial_z\phi}\partial_zf\right) - \partial_y^2\left(\frac1{i\partial_z\phi}f\right).
\end{align*}

Taking \(f(t,z,y) = \chi\left(\lambda_z(z - tv),\lambda_y(y - tv_y)\right)\) we calculate
\[
\partial_tf = - \frac1{2t}(z - tv)\partial_z\chi - \frac1{2t}(y - tv_y)\partial_y\chi - v\partial_z\chi - v_y\partial_y\chi,
\]
and similarly for \(\partial_zf\). Plugging these into the above expression we obtain
\begin{align*}
e^{-i\phi}\tilde{\mc L}\partial_z\left(\frac{e^{i\phi}}{i\partial_z\phi}\chi\right) &= \partial_z\left(\frac1{2t}(z - tv)\chi\right) + \partial_y\left(\frac1{2t}(y - tv_y)\chi\right) + \partial_z\left(\frac12m'(-\partial_z\phi)\partial_z\chi\right)\\
&\quad - \partial_y^2\left(\frac1{i\partial_z\phi}\chi\right) + \err_1,
\end{align*}
where the error term is given by
\begin{align*}
\err_1 &=  \partial_z\left(\frac1{2t}(z - tv)\left(\frac1{i\partial_z\phi}\partial_z\chi\right)\right) + \partial_y\left(\frac1{2t}(y - tv_y)\left(\frac1{i\partial_z\phi}\partial_z\chi\right)\right) \\
&\quad + \partial_z\left(\frac12m'(-\partial_z\phi)\partial_z\left(\frac1{i\partial_z\phi}\partial_z\chi\right)\right) + \frac12m''(- \partial_z\phi)\partial_z^2\phi \partial_z\left(\chi + \frac1{i\partial_z\phi}\partial_z\chi\right)\\
&\quad + i\frac zt\frac{\partial_z^2\phi}{(\partial_z\phi)^2}\partial_z\chi,
\end{align*}
and we have used the fact that \(\partial_z\partial_t\phi = - \frac zt\partial_z^2\phi\).

Returning to the original variables, \((t,z,y)\mapsto (t,x,y)\), we obtain
\begin{align*}
e^{-i\phi}\tilde{\mc L}\Psi_{\mb v} &= - \partial_x\left(\frac1{2t}(z - tv)\chi\right) - \partial_xL_y\left(\frac1{4t^2}(y - tv_y)\chi\right)\\
&\quad + \partial_x\left(\frac12im'(\partial_x\phi)\partial_x\chi\right)  + (\partial_xL_y)^2\left(\frac1{i4t^2\partial_x\phi}\chi\right) + \err_1,
\end{align*}
where we may use the bounds
\begin{gather*}
\partial_x\phi \sim v^{\frac12}\la v\ra^{\frac12},\qquad |\partial_x^2\phi|\lesssim t^{-1}v^{-\frac12}\la v\ra^{\frac12}\\
|\partial_x^\alpha \phi|\lesssim_k t^{-(1 + \alpha)}v^{\frac12 - \alpha}\la v\ra^{-k},\quad \alpha\geq2,
\end{gather*}
to show that whenever \(v\in \Sigma_t\) and \(t\gg1\), the error term \(\err_1\in WP(t,\mb v,t^{-\frac32}v^{-\frac34}\la v\ra^{-\frac14})\).

Finally, we observe that we may expand the leading order terms in this expression to obtain
\[
\tilde{\mc L}\Psi_{\mb v}\in WP(t,\mb v,t^{-1}),
\]
which completes the proof.
\end{proof}
\medskip

\subsection{Testing by wave packets}
We recall that for a given velocity \(\mb v\) the wave packet has similar spatial localization to the hyperbolic part \(u^{\hyp,+}_N\) of \(u\) localized at positive \(x\)-frequency \(N\sim m^{-1}(v)\). From the pointwise estimates of Lemma~\ref{lem:FDR2} we expect that in the region \((x,y)\approx t\mb v\) the leading order part of \(u(t,x,y)\) is given by \(u_N^{\hyp,\pm}(t,x,y)\). As a consequence, we should be able to recover the leading order behavior of \(u\) by testing it against \(\Psi_{\mb v}\).

This heuristic motivates the definition of the function
\[
\gamma(t,\mb v) = \int u_x\bar\Psi_{\mb v}\,dxdy.
\]
Due to the normalization that \(\int \chi\,dxdy = 1\) we then expect that
\[
u(t,t\mb v)\approx 2t^{-1}\la v\ra^{\frac12}\Re\left(e^{i\phi}\gamma(t,\mb v)\right),
\]
where we note that \(t^{-1}\la v\ra^{\frac12} = \lambda_z\lambda_y\). To make this heuristic precise we prove the following lemma:
\medskip
\begin{lem}\label{lem:GammaBounds}
For \(t\gg1\) we have the estimate
\begin{equation}\label{est:BasicGammaBound}
\|\la v\ra^{\frac12}\gamma\chi_{\Sigma_t}\|_{L^\infty_{\mb v}}\lesssim t\|u_x\|_{L^\infty_{x,y}}.
\end{equation}
as well as the estimate for the difference,
\begin{equation}
\|(u_x(t,t\mb v) - 2t^{-1}\la v\ra^{\frac12}\Re(e^{i\phi}\gamma(t,\mb v)))\chi_{\Sigma_t}\|_{L^\infty_{\mb v}} \lesssim t^{-\frac{13}{12}}\|u\|_X.\label{est:DifferenceLInf}
\end{equation}
\end{lem}
\begin{proof}
The pointwise estimate \eqref{est:BasicGammaBound} follows from the fact that for \(v\in \Sigma_t\) we have
\[
\|\Psi_{\mb v}\|_{L^1_{x,y}}\lesssim t\la v\ra^{-\frac12}.
\]

For the pointwise difference \eqref{est:DifferenceLInf} we first define
\[
u_{x,v}^{\hyp,+} = \sum\limits_{m(N)\sim v}\partial_xu_N^{\hyp,+},
\]
and use the pointwise bound \eqref{est:FDR2-ELLIPTIC-2} for the elliptic piece \(u^\el\) as well as the spatial localization of \(u_N^\hyp = 2\Re u_N^{\hyp,+}\) to obtain
\[
\left|u_x(t,t\mb v) - 2\Re u_{x,v}^{\hyp,+}(t,t\mb v)\right|\lesssim t^{-\frac{13}{12}}\|u\|_X.
\]

Next we use the pointwise bound \eqref{est:FDR2-ELLIPTIC-2} for the elliptic piece to obtain
\[
|t^{-1}\la v\ra^{\frac12}\la u_x^\el,\psi_{\mb v}\ra|\lesssim t^{-\frac{13}{12}}.
\]
We may then use the spatial localization of \(\Psi_{\mb v}\) to obtain
\[
|t^{-1}\la v\ra^{\frac12}\la u^\hyp_x - u_{x,v}^\hyp,\Psi_{\mb v}\ra| \lesssim_k t^{-k}\|u\|_X,
\]
and similarly, recalling that \(u_N^{\hyp,-} = \chi_N^\hyp u_N^-\), we have
\[
|\la u_{x,v}^{\hyp,-} ,\Psi_{\mb v}\ra |\lesssim \sum\limits_{m(N)\sim v}\la \partial_xu_N^-,\chi_N^\hyp \Psi_{\mb v}\ra|\lesssim t^{-k}\|u\|_X,
\]
where the rapid decay follows from the fact that \(\chi_N^\hyp\Psi_{\mb v}\in WP(t,\mb v,1)\) is localized at \emph{positive} wavenumbers \(\sim m^{-1}(v)\) up to rapidly decaying tails at scale \(\lambda_z\lesssim t^{-\frac{23}{48}}\). Combining these bounds we obtain
\[
t^{-1}\la v \ra^{\frac12}\left|\gamma(t,\mb v) - \la u_{x,v}^{\hyp,+},\Psi_{\mb v}\ra\right| \lesssim t^{-\frac{13}{12}}\|u\|_X.
\]
Thus it remains to consider the difference
\[
\mathfrak D = \left|e^{-i\phi}u_{x,v}^{\hyp,+}(t,t\mb v) - t^{-1}\la v\ra^{\frac12}\gamma(t,\mb v)\right|
\]

Next we define
\[
w(t,z,y) = e^{-i\phi}u_{x,v}^{\hyp,+}(t,x,y),
\]
and write the difference as
\begin{align*}
\mathfrak D &= w(t,tv,tv_y) - t^{-1}\la v\ra^{\frac12}\int w(t,z,y)\chi(\lambda_z(z - tv),\lambda_y(y - tv_y))\,dzdy\\
&= t^{-1}\la v\ra^{\frac12}\int \left( w(t,tv,tv_y) - w(t,z,y) \right) \chi(\lambda_z(z - tv),\lambda_y(y - tv_y))\,dzdy
\end{align*}
Applying the elliptic estimate \eqref{est:FDR2-Elliptic1} with the frequency localization of \(u_{x,v}^{\hyp,+}\) we obtain the estimates
\begin{gather*}
\|w_z\|_{L^2}\lesssim \|L_+u_{v,x}^{\hyp,+}\|_{L^2}\lesssim t^{-1}v^{-\frac12}\la v\ra^{\frac12}\|u\|_X,\\
\|w_{yy}\|_{L^2}\lesssim t^{-2}\|(L_y\partial_x)^2u_{v,x}^{\hyp,+}\|_{L^2}\lesssim t^{-2}v\la v\ra\|u\|_X,
\end{gather*}
so we may apply the Sobolev estimate \eqref{est:Holder} with \(\alpha = \frac14\) to \(w\) to obtain
\[
|w(t,tv,tv_y) - w(t,z,y)|\lesssim t^{-\frac54}v^{-\frac18}\la v\ra^{\frac78}\left(|z - tv|^{\frac14} + |y - tv|^{\frac14}\right)\|u\|_X.
\]
As a consequence we obtain the estimate
\[
|\mathfrak D|\lesssim t^{-\frac98}v^{-\frac3{16}}\la v\ra^{\frac{15}{16}}\|u\|_X,
\]
which completes the proof of \eqref{est:DifferenceLInf}.
\end{proof}
\medskip

\subsection{The ODE for \(\gamma\)}
From Lemma~\ref{lem:GammaBounds} we see that \(\gamma\) may be used to estimate the size of \(u_x\) up to errors that decay in time. In order to obtain bounds for \(\gamma\) we will treat \(\mb v\) as a fixed parameter and consider the ODE satisfied by \(\gamma\)
\begin{equation}\label{ResonantODE}
\dot\gamma(t,\mb v) = \la (uu_x)_x,\Psi_{\mb v}\ra + \la u_x,\mc L\Psi_{\mb v}\ra.
\end{equation}
For the first of these terms we may use that there are no parallel resonances to show that at least one of the \(u\) terms must be elliptic and have improved decay. For the second term we use the expression \eqref{LinearForm} to see that to leading order \(\mc L\Psi_{\mb v}\) has a divergence-type structure so we may integrate by parts to obtain improved decay. As a consequence we obtain the following lemma:
\medskip
\begin{lem}
If \(u\) is a solution to \eqref{eqn:cc-fd} then for \(t\gg1\) we have the estimate
\begin{equation}
\|\dot\gamma\chi_{\Sigma_t}\|_{L^\infty_{\mb v}} \lesssim t^{-\frac{13}{12}} \|u\|_X\left(1 + \|u\|_X\right).\label{ODEestPwse}
\end{equation}
\end{lem}
\begin{proof}
We start by considering the nonlinear term appearing in \eqref{ResonantODE}. Integrating by parts we obtain
\[
\frac12\la (u^2)_{xx},\Psi_{\mb v}\ra = - \frac12\la ((u^\hyp)^2)_x,\partial_x\Psi_{\mb v}\ra - \la (u^\hyp u^\el)_x,\partial_x\Psi_{\mb v}\ra - \frac12\la ((u^\el)^2)_x,\partial_x\Psi_{\mb v}\ra.
\]
For the second and third terms we may apply the pointwise bounds of Lemma~\ref{lem:FDR2} to obtain
\[
|\la (u^\hyp u^\el)_x,\partial_x\Psi_{\mb v}\ra| + |\la ((u^\el)^2)_x,\partial_x\Psi_{\mb v}\ra|  \lesssim t^{-\frac{13}{12}}\|u\|_{X}^2.
\]

For the remaining term we first use the spatial localization of \(\Psi_{\mb v}\) to replace \(u^\hyp\) by \(u^\hyp_v\), where we recall that
\[
u^\hyp_v = \sum\limits_{\substack{m(N)\sim v}} u_N^\hyp.
\]
Recalling the definition of \(u_N^\hyp = \chi_N^\hyp u_N\) we may write
\[
\la (u^\hyp_v)^2,\Psi_{\mb v}\ra = \sum\limits_{m(N_1),m(N_2)\sim v}\la u_{N_1}u_{N_2},\chi_{N_1}^\hyp\chi_{N_2}^\hyp\Psi_{\mb v}\ra.
\]
We observe that for sufficiently large \(t\gg1\) (independent of \(v\)) the function \(\Theta = \chi_{N_1}^\hyp\chi_{N_2}^\hyp\Psi_{\mb v}\in WP(t,\mb v,1)\) and hence is localized at frequency \(m^{-1}(v)\) up to rapidly decaying tails at scale \(\lambda_z\lesssim t^{- \frac{23}{48}}\). In particular, for \(t\gg1\) sufficiently large (independently of \(t,v\)) we have
\[
|P_{< \frac12m^{-1}(v)}\Theta |\lesssim_k t^{-k},\qquad |P_{\geq \frac32 m^{-1}(v)}\Theta|\lesssim_k t^{-k}.
\]
However, the product \(u_{N_1}u_{N_2}\) has compact Fourier support in neighborhoods of size \(O(\delta m^{-1}(v))\) about the frequencies \(0, \pm 2m^{-1}(v)\). In particular, by choosing \(0<\delta\ll1\) sufficiently small (independently of \(v\)) we may ensure that
\[
P_{\frac12 m^{-1}(v)\leq \cdot < \frac32 m^{-1}(v)}(u_{N_1}u_{N_2}) = 0,
\]
and hence
\[
|\la (u^\hyp_v)^2,\Psi_{\mb v}\ra|\lesssim t^{-k}\|u\|_X^2.
\]

To complete the estimate we consider the linear term. We first recall from Lemma~\ref{lem:Microlocal} that \(\mc L\Psi_{\mb v}\in WP(t,\mb v,t^{-1})\) and hence satisfies \(\|\la v\ra^{\frac12}\mc L\Psi_{\mb v}\|_{L^1_{x,y}}\lesssim 1\). Estimating as in Lemma~\ref{lem:GammaBounds} we then obtain
\[
|\la u_x^\hyp - u_{x,v}^{\hyp,+},\mc L\Psi_{\mb v}\ra|\lesssim_k t^{-k}\|u\|_X.
\]
Next we recall the expression \eqref{LinearForm} for the operator \(e^{-i\phi}\mc L\Psi_{\mb v}\). In particular, we may take \(w = e^{-\phi}u_{x,v}^{\hyp,+}(t,x,y)\) as before and integrate by parts to obtain
\begin{align*}
\la u_{x,v}^{\hyp,+},\mc L\Psi_{\mb v}\ra &= \la w_z,\frac1{2t}(z - tv)\chi\ra + \la w_y,\frac1{4t^2}(y - tv_y)\chi\ra\\
&\quad - \la w_z,\frac12im'(\partial_x\phi) \partial_x\chi\ra + \la w_{yy},\frac1{4it^2\partial_x\phi}\chi\ra + \la w,\err\ra,
\end{align*}
where the error term \(\err\in WP(t,\mb v,t^{-\frac32}v^{-\frac34}\la v\ra^{-\frac14})\). Applying the \(L^2\)-estimates for \(w\) as in Lemma~\ref{lem:GammaBounds} and the pointwise estimate \eqref{est:FDR2-HYPERBOLIC-2} for the final term, we obtain the estimate \eqref{ODEestPwse}.
\end{proof}
\medskip

\subsection{Proof of global existence}~
We now complete the proof of Theorem~\ref{thm:Main}. We choose \(\mc T_0\geq 1\) and by taking \(\mc M\gg1\) sufficiently large and \(0<\epsilon\ll1\) sufficiently small we may find a solution \(S(-t)u\in C([0,T];X(0))\) to \eqref{eqn:cc-fd} for some \(T\geq \mc T_0\). Next we assume that the bootstrap assumption \eqref{BS} holds on the interval \([0,T]\) from which we obtain the energy estimate \eqref{AP}.

Next we use the estimate \eqref{est:BasicGammaBound} to bound \(\gamma\) at time \(\mc T_0\) in terms of \(\|u_x\|_{L^\infty}\) and the Sobolev estimate \eqref{Starter410} to obtain
\[
|\gamma(\mc T_0,\mb v)|\lesssim \epsilon \mc T_0^{C\mc M\epsilon}.
\]
We may then solve the ODE satisfied by \(\gamma\) on the time interval \([\mc T_0,T]\) using the estimate \eqref{ODEestPwse} to obtain
\[
\|\la v\ra^{\frac12}\gamma \chi_{\Sigma_t}\|_{L^\infty_{\mb v}}\lesssim \epsilon \mc T_0^{C\mc M\epsilon} + \int_{\mc T_0}^t \|\la v\ra^{\frac12}\dot\gamma \chi_{\Sigma_t}\|_{L^\infty_{\mb v}}\lesssim \epsilon \mc T_0^{C\mc M\epsilon} + \epsilon \mc T_0^{-\frac1{12} + 2C\mc M \epsilon},
\]
provided \(0<\epsilon\ll1\) is sufficiently small. We may then apply the estimate \eqref{est:DifferenceLInf} for the difference between \(u_x\) and \(2t^{-1}\la v\ra^{\frac12}\Re(e^{i\phi}\gamma)\) to obtain
\[
\|u_x\chi_{\Sigma_t}\|_{L^\infty}\lesssim \epsilon t^{-1}\left(\mc T_0^{C\mc M\epsilon} + \mc T_0^{-\frac1{12} + 2C\mc M \epsilon}\right).
\]
By choosing \(\mc M\gg1\) sufficiently large and \(0<\epsilon\ll1\) sufficiently small we may combine this with the estimate \eqref{ImprovedDecay} for \(u_x\) in the region \(\Sigma_t^c\) to obtain the bound
\[
\|u_x\|_{L^\infty}\leq \frac12\mc M\epsilon t^{-\frac12}\la t\ra^{-\frac12},
\]
which closes the bootstrap. The solution \(u\) then exists globally and satisfies the energy estimate \eqref{Fdn} and the pointwise estimate \eqref{est:PTWISEDECAY}.

\subsection{Proof of scattering}~
It remains to prove that our solutions scatters in \(L^2\). As in~\cite{2014arXiv1409.4487H} we do not have  scattering in the sense that \(\mc L u\in L^1_tL^2_{x,y}\) but we are able to construct a normal form correction to remove the worst bilinear interactions and show that \(S(-t)u(t)\) converges in \(L^2\) as \(t\rightarrow\infty\). We note that for translation invariant initial data the worst nonlinear interactions are the high-low interactions (see Appendix~\ref{app:IP}). However, the spatial localization ensures that these interactions can only occur on very short timescales, thus attenuating their effect. From the pointwise and elliptic estimates of Section~\ref{sect:KS} we see that the worst nonlinear interactions for spatially localized initial data are the high-high (hyperbolic) interactions for which we may construct a well-defined normal form.

We first define the leading order part of \(w\) by
\[
w = P_{t^{-\frac16}<\cdot\leq t^{\frac1{12}}}u,
\]
and then have the following lemma:
\medskip
\begin{lem}
For \(t\gg1\) we have the estimate
\begin{equation}\label{Reduction}
\|uu_x - 2\Re(w^+w_x^+)\|_{L^2}\lesssim t^{-\frac{97}{96}}\|u\|_X^2.
\end{equation}
\end{lem}
\begin{proof}
We start by using the estimate \eqref{ImprovedDecay} to reduce the estimate to the region \(\Sigma_t\),
\[
\|uu_x\chi_{\Sigma_t^c}\|_{L^2}\lesssim \|u_x \chi_{\Sigma_t^c}\|_{L^\infty}\|u\|_{L^2}\lesssim t^{-\frac{97}{96}}\|u\|_X^2.
\]
Next we use the pointwise estimates of Lemma~\ref{lem:FDR2} to reduce to the hyperbolic parts,
\[
\|(uu_x - u^\hyp u^\hyp_x) \chi_{\Sigma_t}\|_{L^2}\lesssim\|u_x^\el\|_{L^\infty} \|u\|_{L^2} + \|u_\el\chi_{\Sigma_t}\|_{L^\infty}\|u_x^\hyp\|_{L^2}\lesssim t^{-\frac{97}{96}}\|u\|_X^2.
\]

We observe that
\[
u^\hyp u_x^\hyp = 2\Re(u^{\hyp,+}u_x^{\hyp,+}) + \partial_x|u^{\hyp,+}|^2,
\]
and that
\[
\partial_x|u^{\hyp,+}|^2 = 2\Im (\bar u^{\hyp,+}L_z^+u^{\hyp,+}),
\]
so applying the elliptic estimate \eqref{est:FDR2-Elliptic1} we obtain
\[
\|(u^\hyp u_x^\hyp - 2\Re(u^{\hyp,+}u_x^{\hyp,+}))\chi_{\Sigma_t}\|_{L^2} \lesssim \|u^\hyp \chi_{\Sigma_t}\|_{L^\infty}\|L_z^+u^{\hyp,+}\chi_\hyp^+\|_{L^2}\lesssim t^{-\frac{97}{96}}\|u\|_X^2.
\]

Taking \(w^{\hyp,+} = \sum\limits_N \chi_N^\hyp w_N^+\) as above we see that
\[
w^{\hyp,+} \chi_{\Sigma_t} = u^{\hyp,+} \chi_{\Sigma_t},
\]
and hence
\[
\|uu_x - 2\Re(w^{\hyp,+} w_x^{\hyp,+})\|_{L^2} \lesssim t^{-\frac{97}{96}}\|u\|_X^2.
\]

Finally we may once again apply the pointwise estimates of Lemma~\ref{lem:FDR2} to obtain the bound
\[
\|w^+w_x^+ - w^{\hyp,+}w_x^{\hyp,+}\|_{L^2}\lesssim t^{-\frac{97}{96}}\|u\|_X^2,
\]
which completes the proof.
\end{proof}
\medskip

We now construct a normal form for the nonlinear term \(2\Re(w^+w_x^+)\). Here we essentially proceed as in Proposition~\ref{prop:NRG} and define a symmetric bilinear form \(B[u,v]\) with symbol
\[
b(\K_1,\K_2) = \frac{\xi_1 + \xi_2}{2\Omega(\K_1,\K_2)},
\]
where \(\Omega\) is the resonance function defined as in \eqref{ResonanceFunction}. By construction we have
\[
\mc L B[f,g] = \frac12(fg)_x + B[f,\mc Lg] + B[\mc L f,g].
\]
Further, we have the following lemma:
\medskip
\begin{lem}
We have the estimates
\begin{align}
\|B[w^+,w^+]\|_{L^2} &\lesssim t^{-\frac23}\|u\|_X^2,\label{AMC}\\
\|B[w_+,\mc L w_+]\|_{L^2} &\lesssim t^{-\frac32}\|u\|_X^2(1 + \|u\|_X).\label{BMC}
\end{align}
\end{lem}
\begin{proof}
We note that here we need only consider high frequency outputs as \(\xi_1,\xi_2>0\) have the same sign. From the estimate \eqref{FDResonanceLB} we see that for \(0<\xi_1\lesssim \xi_2\), the symbol \(b\) satisfies the bounds
\[
b(\K_1,\K_2)\lesssim \frac{\la \xi_2\ra}{\xi_1\xi_2},
\]
and hence we may decompose
\[
B[u^+,v^+] = Q[\partial_x^{-1}u^+,v^+] + Q[v^+,\partial_x^{-1}u^+],
\]
where \(Q\) is given by
\[
Q[u^+,v^+] = \sum\limits_{N}B[u_{<N}^+,v_N^+].
\]
We may then verify that the corresponding symbol \(q\in \symb\) and applying the Coifman-Meyer Theorem \eqref{CM} with the frequency localization of \(w^+\) we obtain the estimates
\begin{align*}
\|B[w^+,w^+]\|_{L^2}&\lesssim \|w_+\|_{L^\infty}\|\partial_x^{-1}w_+\|_{L^2},\\
\|B[w_+,\mc L w_+]\|_{L^2} &\lesssim \|w_+\|_{L^\infty}\|\partial_x^{-1}\mc Lw^+\|_{L^2}.
\end{align*}

For the estimate \eqref{AMC} we may use the pointwise estimates of Lemma~\ref{lem:FDR2} and the frequency localization to obtain
\[
\|w^+\|_{L^\infty}\lesssim t^{-\frac34}\|u\|_X,\qquad \|\partial_x^{-1}w^+\|_{L^2}\lesssim t^{\frac16}\|u\|_X.
\]

For the estimate \eqref{BMC} we instead compute
\[
\partial_x^{-1}\mc L w^+ = \frac12  P_{t^{-\frac16}<\cdot\leq t^{\frac1{12}}}^+(u^2) + \partial_x^{-1}[\partial_t,P_{t^{-\frac16}<\cdot\leq t^{\frac1{12}}}^+]u.
\]
Using the frequency localization we then obtain
\[
\|\partial_x^{-1}\mc L w^+\|_{L^2}\lesssim \|u\|_{L^\infty}\|u\|_{L^2} + t^{-\frac56}\|u\|_{L^2}\lesssim t^{-\frac34}\|u\|_X^2 + t^{-\frac56}\|u\|_X,
\]
which completes the proof.
\end{proof}
\medskip

To complete the proof of scattering we apply the estimates \eqref{Reduction}, \eqref{BMC} with the energy estimate \eqref{Fdn} to obtain the bound
\[
\|\mc L(u - 2\Re B[w^+,w^+])\|_{L^2}\lesssim t^{-\frac{97}{96} + C\epsilon}\epsilon^2.
\]
In particular, provided \(0<\epsilon\ll1\) is sufficiently small, we can see that given the integrability in time of the nonlinear interactions we can construct a Cauchy sequence for $S(-t)(u - 2\Re B[w^+,w^+])$ in $L^2$ converging to a  \(W\in L^2\) so that for \(t\gg1\),
\[
\|S(-t) (u - 2\Re B[w^+,w^+]) -  W\|_{L^2}\lesssim t^{-\frac1{96} + C\epsilon}\epsilon^2.
\]
Applying the estimate \eqref{AMC} we have
\[
\|B[w^+,w^+]\|_{L^2}\lesssim t^{-\frac34 + C\epsilon}\epsilon^2, 
\]
and hence \(u\) satisfies the estimate \eqref{Fds}. Finally we note that \(\|W\|_{L^2} = \|u_0\|_{L^2}\) by conservation of mass.

\begin{appendix}

\section{Ill-posedness in Besov-type spaces}\label{app:IP}
In this section we show that the infinite depth equation \eqref{eqn:cc} is ill-posed in (almost all) the natural Galilean-invariant, scale-invariant Besov-type refinements of \(\dot H^{\frac14,0}\) considered in \cite{2016arXiv160806730K}. 

To define these spaces we make an almost orthogonal decomposition
\[
u = \sum_{N\in2^\Z}\sum_{k\in\Z}u_{N,k},
\]
where each \(u_{N,k}\) has Fourier-support in the trapezium
\[
\mc Q_{N,k} = \left\{(\xi,\eta)\in\R^2: \frac12 N<|\xi|<2N,\ \left|\frac\eta\xi - kN^{\frac12}\right| < \frac34 N^{\frac12}\right\}.
\]
We then define the space \(\ell^q\ell^p L^2\) with norm
\[
\|u\|_{\ell^q\ell^p L^2}^q = \sum\limits_{N\in 2^\Z} N^{\frac14 q}\left(\sum\limits_{k\in \Z}\|u_{N,k}\|_{L^2}^p\right)^{\frac qp}.
\]
It is straightforward to verify that these spaces are indeed both scale-invariant and Galilean invariant by recalling that the Galilean shift \eqref{Galilean} corresponds to the map
\[
\hat u(t,\xi,\eta)\mapsto e^{-ic^2t\xi}e^{2ict\eta}\hat u(t,\xi,\eta - c\xi).
\]
Further, it is clear that when \(p = q = 2\) we have \(\ell^2\ell^2L^2 = \dot H^{\frac14,0}\). We remark that analogously to \cite[Theorem~1.4]{2016arXiv160806730K} we may show that \(\ell^q\ell^pL^2\) embeds continuously into the space of distributions whenever \(1\leq q\leq\infty\) and \(1\leq p<\frac43\) and that it contains the Schwartz functions for all \(p > 1\).

% \comment{(Jeremy) For our own benefit, but likely we should comment this out of the final version:
% \begin{rem} In \cite{2016arXiv160806730K}, the authors also address the issue Schwartz functions in \(\ell^q\ell^pL^2\).
% The calculation there uses equivalence of $\ell^p$ and $\ell^2$ norms in $\RR^k$, which occurs with constants of the form:
% \[
% k^{\frac{2-p}{2p}}.
% \]
% Hence, in their estimate on page $26$, proving Theorem $1.4$, $(iii)$, we observe that $k = \frac{r^2}{\lambda^4} $, giving for the dependence upon $r$, $\lambda$ the quantity
% \[
% \lambda^{\frac12} \left( \frac{r^2}{\lambda^4} \right)^{\frac{2-p}{2p}} \| \phi_{\lambda \cap B} \|_{L^2} = r^{\frac{2}{p} - 1} \lambda^{\frac52} \lambda^{- \frac{ 4}{p}} (r^2 \lambda)^{\frac12} = r^{\frac{2}{p}} \lambda^{3 - \frac{4}{p}} .
% \]
% In our situation, we have
% \[
% \lambda^{\frac14} \left( \frac{r}{\lambda^{\frac32}} \right)^{\frac{2-p}{2p}} \| \phi_{\lambda \cap B} \|_{L^2} = r^{\frac{1}{p} - \frac12} \lambda^{\frac54 } \lambda^{- \frac{ 3}{2p}} (r \lambda)^{\frac12} = r^{\frac{1}{p} } \lambda^{\frac32 - \frac{3}{2p}}
% \]
% giving that any $p > 1$ would avoid the moment condition. We can also take \(p = 1\) if \(q = \infty\).
% \end{rem}
% }

Using similar ideas to \cite{2016arXiv160806730K,MR1885293}, we then obtain the following ill-posedness result:
\medskip
\begin{thm}\label{thm:IP}
Let \(1\leq q\leq\infty\) and \(1<p\leq\infty\). Then there does not exist a continuously embedded space \(X_T\subset C([-T,T]:\ell^q\ell^pL^2)\) so that for all \(\phi\in \ell^q\ell^pL^2\),
\begin{gather}
\|S_\infty(t)\phi\|_{X_T}\lesssim \|\phi\|_{\ell^q\ell^pL^2},\\
\left\|\int_0^tS_\infty(t - t')[u(t')\ \partial_xu(t')]\,dt'\right\|_{X_T}\lesssim \|u\|_{X_T}^2,
\end{gather}
where \(S_\infty(t)\) is the infinite depth linear propagator, defined as in \eqref{Propagator}.

In particular, for the infinite depth equation \eqref{eqn:cc}, the solution map \(u_0\mapsto u(t)\) (considered as a map on \(\ell^q\ell^pL^2\)) fails to be twice differentiable at \(u_0 = 0\).
\end{thm}
\begin{proof}
We proceed by contradiction. Suppose that such a space \(X_T\) does exist, then for any \(\phi\in \ell^q\ell^pL^2\) and \(t\in[-T,T]\) we have the estimate
\begin{equation}\label{est:BoundToFail}
\left\|\int_0^tS_\infty(t - t')\partial_x[(S_\infty(t')\phi)^2]\,dt'\right\|_{\ell^q\ell^pL^2}\lesssim \|\phi\|_{\ell^q\ell^pL^2}^2.
\end{equation}
Our goal is to show that this estimate must fail for a suitable choice of \(\phi\). We note that as we will only work with \(O(1)\) choices of \(x\)-frequency, our argument is independent of the choice of \(q\).

We first choose low and high frequency parameters \(0<\delta\ll 1\ll N\), where for convenience we assume that both are dyadic integers. We then define the high and low frequency sets by
\begin{align*}
E_\high &:= \left\{\K\in\R^2:- \frac14\delta<|\xi| - N<\frac14\delta,\ \left|\frac\eta\xi\right|< N^{\frac12}\right\},\\
E_\low &:= \left\{\K\in\R^2:-\frac14\delta<|\xi| - \delta<\frac14\delta,\ \left|\frac\eta\xi\right|< N^{\frac12}\right\}.
\end{align*}
We observe that
\[
|E_\high|\sim\delta N^{\frac32},\qquad  |E_\low|\sim \delta^2N^{\frac12}
\]
Further, an elementary algebraic calculation gives us that,
\begin{align*}
E_\high + E_\low &\subset \left\{\K\in\R^2:- \frac12\delta<|\xi| - N \mp \delta<\frac12\delta,\ \left|\frac\eta\xi\right|< 10 N^{\frac12}\right\},\\
E_\high + E_\low&\supset \left\{\K\in\R^2:- \frac12\delta<|\xi| - N \mp \delta<\frac12\delta,\ \left|\frac\eta\xi\right|< \frac1{10}N^{\frac12}\right\},
\end{align*}
as well as
\[
(E_\high + E_\low)\cap (E_\low + E_\low) = \emptyset,\qquad (E_\high + E_\low)\cap(E_\high + E_\high) = \emptyset.
\]

We define functions associated to the sets \(E_\high,E_\low\) by
\[
\hat\phi_\high = \delta^{-\frac12}N^{-1}\mb1_{E_\high},\qquad \hat\phi_\low = \delta^{\frac1{2p}-\frac32}N^{-\frac1{2p}}\mb1_{E_\low},
\]
and take
\[
\phi = \phi_\high + \phi_\low.
\]
As the high frequency set \(E_\high\subset\mc Q_{N,0}\) we have the estimate
\begin{equation}\label{est:HiBound}
\|\phi_\high\|_{\ell^q\ell^pL^2}\sim 1.
\end{equation}
As the low frequency set
\(
E_\low\subset \bigcup_{|k|\leq \delta^{-\frac12}N^{\frac12}}\mc Q_{\delta,k}
\)
we obtain the low frequency bound
\begin{equation}\label{est:LoBound}
\|\phi_\low\|_{\ell^q\ell^pL^2}\sim 1.
\end{equation}
Combining \eqref{est:HiBound} and \eqref{est:LoBound} with the fact that \(E_\high\cap E_\low = \emptyset\) we obtain
\begin{equation}
\|\phi\|_{\ell^q\ell^pL^2}\sim 1.
\end{equation}
Further, from the bounds on \(|E_\high + E_\low|\) and \(|E_\low|\) we obtain the convolution estimate
\[
\|\hat\phi_\high*\hat\phi_\low\|_{L^2} \gtrsim \delta^{\frac12 + \frac1{2p}}N^{\frac14 - \frac1{2p}}.
\]

We now consider the left hand side of \eqref{est:BoundToFail}. Using the support properties of the sums of \(E_\high + E_\high\), \(E_\low + E_\low\) and \(E_\high + E_\low\) we obtain the lower bound
\[
\left\|\int_0^tS_\infty(t - s)\partial_x[(S_\infty(s)\phi)^2]\,ds\right\|_{\ell^q\ell^pL^2} \gtrsim N^{\frac54}\left\|\int_0^t\int_{\K = \K_1 + \K_2}e^{is\Omega(\K_1,\K_2)}\hat\phi_\high(\K_1)\hat\phi_\low(\K_2)\,d\K_1\,ds\right\|_{L^2},
\]
where the resonance function \(\Omega = \Omega_\infty\) is defined as in \eqref{ResonanceInf}. If \(\K_1\in E_\high\), \(\K_2\in E_\low\) and \(\K = \K_1 + \K_2\) then provided \(N\gg 1\) we obtain the estimate
\[
|\Omega(\K_1,\K_2)|\lesssim N\delta.
\]

If we choose \(\delta = N^{-(1 + \epsilon)}\) then for \(|t|\sim 1\) we obtain
\[
\left|\int_0^t e^{it'\Omega(\K_1,\K_2)}\,dt'\right| = \left|\frac{e^{it\Omega(\K_1,\K_2)} - 1}{i\Omega(\K_1,\K_2)}\right| \gtrsim 1 + O(N^{-\epsilon}).
\]
As \(\hat\phi_\high,\hat\phi_\low\geq0\) then for \(N\gg1\) we obtain
\[
\left\|\int_0^t\int_{\K = \K_1 + \K_2}e^{is\Omega(\K_1,\K_2)}\hat\phi_\high(\K_1)\hat\phi_\low(\K_2)\,d\K_1\,ds\right\|_{L^2}\gtrsim \|\phi_\high*\phi_\low\|_{L^2},
\]
and hence
\[
\left\|\int_0^tS_\infty(t - s)\partial_x[(S_\infty(s)\phi)^2]\,ds\right\|_{\ell^q\ell^pL^2}\gtrsim N^{1 - \frac1p}N^{- \epsilon(\frac12 + \frac1{2p})}.
\]
If \(p>1\) then by choosing \(0<\epsilon\ll1\) sufficiently small we may take \(N\rightarrow+\infty\) to obtain a contradiction.
\end{proof}

\end{appendix}

\bibliographystyle{abbrv}
\bibliography{refs}

\end{document}